\documentclass[12pt]{amsart}
\usepackage{tikz-cd}
\usepackage{enumitem}
\setlist[description]{leftmargin=\parindent,labelindent=\parindent}
\usetikzlibrary{matrix,arrows,decorations.pathmorphing}
\usepackage{amssymb}
\usepackage{amsmath,amscd}
\usepackage{mathrsfs}
\usepackage{url}
\usepackage{cite}
\usepackage{fullpage}
\usepackage{hyperref}
\usepackage{verbatim}
\usepackage{comment}
\usepackage{fancyvrb}
\usepackage{fvextra}
\usepackage{caption}
\usepackage{ytableau}

\usepackage{pgfplots}
\pgfplotsset{compat=1.16}

\newtheorem{thm}{Theorem}
\newtheorem{prop}[thm]{Proposition}
\newtheorem{lem}[thm]{Lemma}
\newtheorem{cor}[thm]{Corollary}
\newtheorem{conj}[thm]{Conjecture}

\theoremstyle{definition}
\newtheorem{definition}[thm]{Definition}

\newtheorem{rem}[thm]{Remark}

\newcommand{\X}{\mathcal{X}}

\renewcommand{\H}{\mathsf{H}}

\renewcommand{\ss}{\mathbb{S}}
\newcommand{\qq}{\mathbb{Q}}
\newcommand{\R}{\mathsf{R}}

\newcommand{\RH}{\mathsf{RH}}

\newcommand{\M}{\mathcal{M}}

\newcommand{\A}{\mathsf{A}}

\renewcommand{\S}{\mathsf{S}}

\DeclareMathOperator{\ct}{ct}

\renewcommand{\tilde}{\widetilde}
\DeclareMathOperator{\rank}{rank}

\DeclareMathOperator{\im}{Im}

\newcommand{\FZ}{\mathsf{FZ}}

\usepackage{wrapfig}

\renewcommand{\bar}{\overline}
\newcommand{\Mb}{\overline{\M}}

\title{The Gorenstein property and Pixton's conjecture for compact type moduli}
\author{Samir Canning}
\author{Hannah Larson}
\author{Johannes Schmitt}

\thanks{S.C. was supported by a Hermann-Weyl-instructorship from the Forschungsinstitut für Mathematik at ETH Z\"urich and the SNSF Ambizione grant 223473. This research was partially conducted during the period H.L. served as a Clay Research Fellow. J.S. was supported by the SNSF grant 219369 and SwissMAP}

\begin{document}
\begin{abstract}
%We completely determine for which pairs $(g,n)$ the tautological ring of $\M_{g,n}^{\ct}$ is Gorenstein. 
We show that the tautological ring of $\M_{g,n}^{\ct}$ is not Gorenstein for $g\geq 2$ and $2g+n\geq 12$. We prove new cases of Pixton's conjecture that the $3$-spin relations are a complete set of relations for the tautological ring, including $\M_{6}^{\ct}$, $\M_{5,2}^{\ct}$, and $\M_7^{\ct}$. These are the first known cases where Pixton's conjecture is true, but the tautological ring is not Gorenstein. These results are also a key ingredient in recent work on non-tautological cycles on the moduli space of principally polarized abelian varieties.
\end{abstract}
\maketitle
\section{Introduction}
\subsection{The tautological ring}
Let $\Mb_{g,n}$ be the moduli space of stable curves of genus $g$ with $n$ markings. For a stable graph $\Gamma$ of genus $g$ with $n$ legs and vertex set $V$, we set \[\Mb_{\Gamma}=\prod_{v\in V} \Mb_{g(v),n(v)}.\]
There is a proper gluing morphism
\[
\xi_{\Gamma}:\Mb_{\Gamma}\rightarrow \Mb_{g,n}
\]
whose image is the closure of the locus of curves with dual graph $\Gamma$. Let $\pi:\Mb_{g,n+1}\rightarrow \Mb_{g,n}$ be the universal curve and $s_i:\Mb_{g,n}\rightarrow \Mb_{g,n+1}$ the universal sections. The $\psi$ classes on $\Mb_{g,n}$ are defined as
\[
\psi_i=c_1(s_i^*\omega_{\pi})\in \mathsf{A}^1(\Mb_{g,n}).
\]
The Arbarello--Cornalba $\kappa$ classes are 
\[
\kappa_j = \pi_*(\psi_{n+1}^{j+1})\in \mathsf{A}^j(\Mb_{g,n}),
\]
and the lambda classes are
\[
\lambda_k = c_k (\pi_*\omega_\pi) \in \A^k(\Mb_{g,n}).
\]

Let $\pi_v:\Mb_{\Gamma}\rightarrow \Mb_{g(v),n(v)}$ be the projection map. A \emph{decoration} on $\Gamma$ is a product of monomials in $\psi$ and $\kappa$ classes of total degree at most $3g(v)-3+n(v)$ pulled back from the moduli spaces associated to the vertices by $\pi_v$. Let $\S^*(\Mb_{g,n})$ be the $\qq$-vector space whose basis elements are pairs $[\Gamma,\gamma]$, where $\Gamma$ is a stable graph of genus $g$ with $n$ legs and $\gamma$ is a decoration. The vector space $\S^*(\Mb_{g,n})$ is finite dimensional and graded:
\[
\deg\ [\Gamma,\gamma] = |E(\Gamma)|+\deg(\gamma),
\]
where $E(\Gamma)$ is the set of edges of $\Gamma$.
It also has the structure of a $\qq$-algebra, with the product determined by the intersection theory of the boundary strata of $\Mb_{g,n}$, see \cite[Appendix]{GraberPandharipande} for details. The product respects the grading, and so $\S^*(\Mb_{g,n})$ is a graded algebra called the \emph{strata algebra}.

Let $\M_{g,n}^{\ct}$ be the moduli space of compact type curves of genus $g$ with $n$ markings. For $g \geq 1$, there is a degree shifting map on underlying graded vector spaces
\[
\xi_{\mathsf{loop}}:\S^{*-1}(\Mb_{g-1,n+2})\rightarrow \S^*(\Mb_{g,n}), 
\]
given by connecting the last two legs of the stable graph in the domain by an edge. The quotient is denoted by $\S^*(\M^{\ct}_{g,n})$, the \emph{compact type strata algebra}. It has as a basis the pairs $[\Gamma,\gamma]$, where $\Gamma$ is a stable tree of genus $g$ with $n$ legs. 

%\johannes{Small nitpick: Pixton's program would also eliminate any compact type decorated strata that vanish for dimension reasons (decoration on vertex $v$ above the local socle degree at $v$). We might add a remark about this in the computation-section; this comment is mostly an FYI and reminder for myself. }

Strata algebra classes give rise to Chow and cohomology classes on $\Mb_{g,n}$ by pushforward from $\Mb_{\Gamma}$. The following diagram commutes:
\[
\begin{tikzcd}
{\S^*(\Mb_{g,n})} \arrow[r] \arrow[d] & {\mathsf{A}^*(\Mb_{g,n})} \arrow[d] \\
{\S^*(\M_{g,n}^{\ct})} \arrow[r]      & {\mathsf{A}^*(\M^{\ct}_{g,n})}.    
\end{tikzcd}   
\]
The images of the horizontal maps are by definition the \emph{tautological rings} $\R^*(\Mb_{g,n})$ and $\R^*(\M_{g,n}^{\ct})$, respectively. We denote the kernel of the upper horizontal map by $\mathsf{I}_{\A}^*(\Mb_{g,n})$ and of the lower horizontal map by $\mathsf{I}_{\A}^*(\M_{g,n}^{\ct})$. These are the \emph{ideals of tautological relations} for stable and compact type moduli, respectively.

We can further compose with the cycle class map, obtaining a commutative diagram:
\[
\begin{tikzcd}
{\S^*(\Mb_{g,n})} \arrow[r] \arrow[d] & {\mathsf{H}^{2*}(\Mb_{g,n})} \arrow[d] \\
{\S^*(\M_{g,n}^{\ct})} \arrow[r]      & {\mathsf{H}^{2*}(\M^{\ct}_{g,n})}.     
\end{tikzcd}
\]
The images of the horizontal maps are the \emph{tautological cohomology rings} $\RH^{2*}(\Mb_{g,n})$ and $\RH^{2*}(\M_{g,n}^{\ct})$. The kernel of the upper horizontal map is denoted by $\mathsf{I}_{\H}^*(\Mb_{g,n})$ and of the lower horizontal map by $\mathsf{I}_{\H}^*(\M_{g,n}^{\ct})$. These are the \emph{ideals of cohomological tautological relations} for stable and compact type moduli, respectively. We always have $\mathsf{I}_{\A}^*\subset \mathsf{I}_{\H}^*$.

A fundamental open problem in the intersection theory of moduli spaces of curves is to determine the structure of the tautological rings. Even basic aspects are not well understood. It is unknown, for example, if $\mathsf{I}^*_{\A}=\mathsf{I}^*_{\H}$ for stable or compact type moduli. It is also unknown if the restriction map $\mathsf{I}_{\A}^*(\Mb_{g,n})\rightarrow \mathsf{I}_{\A}^*(\M_{g,n}^{\ct})$ is surjective, and analogously for the cohomological ideals. In other words, we do not know if the sequence
\[
\R^{k-1}(\Mb_{g-1,n+2})\rightarrow \R^k(\Mb_{g,n})\rightarrow \R^{k}(\M^{\ct}_{g,n})\rightarrow 0
\]
is always exact.

There have been two prominent proposals for the structure of the tautological rings: the Gorenstein property and Pixton's conjecture, which we describe in the next two sections.
\subsection{The Gorenstein property}
The Gorenstein property was first studied by Faber in the tautological ring of $\M_g$. The tautological ring $\R^*(\M_g)$ is the image of $\R^*(\Mb_{g})$ in $\mathsf{A}^*(\M_g)$ under restriction, or equivalently, the subring generated by the $\kappa$ classes. 

The Gorenstein property concerns the intersection pairing
\begin{equation}\label{smoothpairing}
    \R^i(\M_g)\times \R^{g-2-i}(\M_g)\rightarrow \R^{g-2}(\M_g)\cong \qq.
\end{equation}
Looijenga \cite{Looijenga} showed that $\dim \R^{g-2}(\M_g)\leq 1$ and that $\dim \R^i(\M_g)=0$ for $i>g-2$. Faber \cite[Theorem 2]{Faberconjecture} proved that $\dim \R^{g-2}(\M_g)\geq 1$, establishing the isomorphism
\[
\R^{g-2}(\M_g)\cong \qq.
\]
The pairing \eqref{smoothpairing} is defined as
\[
(\alpha,\beta)\mapsto \int_{\Mb_g} \bar{\alpha}\cdot\bar{\beta}\cdot\lambda_{g-1}\cdot\lambda_{g}.
\]
Here, $\bar{\alpha}$ and $\bar{\beta}$ are arbitrary lifts of $\alpha$ and $\beta$ to $\R^*(\Mb_{g})$. The pairing is well-defined because $\lambda_{g-1}\cdot\lambda_{g}$ vanishes on the boundary $\Mb_{g}\smallsetminus \M_g$ \cite[p. 112]{Faberconjecture}. There is also an analogous pairing in the cohomological tautological ring.

Faber conjectured \cite[Conjecture 1a]{Faberconjecture} that the intersection pairings \eqref{smoothpairing}
are perfect for all $i$. That is, $\R^*(\M_g)$ is a Gorenstein ring with socle in codimension $g-2$. If the Gorenstein conjecture is true, then there is an algorithm to compute the ideal of relations among the $\kappa$ classes: a homogeneous polynomial in the $\kappa$ classes of degree $i$ is zero if and only if it pairs to zero with all homogeneous $\kappa$ polynomials of degree $g-2-i$. The pairing can be computed explicitly using the proportionalities of \cite[Conjecture 1c]{Faberconjecture}, which has now been proven in many different ways \cite{GetzlerPandharipande,LiuXu,GouldenJacksonVakil}.

Faber provided low genus evidence for the Gorenstein conjecture on $\M_g$ computationally. Originally, he showed that the conjecture is true when $g\leq 15$ \cite{Faberconjecture}, and later he extended these computations to $g\leq 23$. He also showed, however, that the ring generated by the $\kappa$ classes modulo the Faber--Zagier relations is not Gorenstein when $g=24$. See Section \ref{pixconj} below for a discussion of the Faber--Zagier relations and their generalizations. For $g\geq 24$, the Gorenstein conjecture remains open.

Analogous conjectures (or speculations) were made for $\Mb_{g,n}$ and $\M_{g,n}^{\ct}$ \cite{Faberspeculation,PandharipandeGorenstein}. There are intersection pairings
\begin{equation}\label{barpairing}
    \R^i(\Mb_{g,n})\times \R^{3g-3+n-i}(\Mb_{g,n})\rightarrow \R^{3g-3+n}(\Mb_{g,n})\cong \qq, \qquad (\alpha,\beta)\mapsto \int_{\Mb_{g,n}} \alpha\cdot \beta
\end{equation}
and
\begin{equation}\label{ctpairing}
    \R^i(\M_{g,n}^{\ct})\times \R^{2g-3+n-i}(\M_{g,n}^{\ct})\rightarrow \R^{2g-3+n}(\M_{g,n}^{\ct})\cong \qq, \qquad (\alpha,\beta)\mapsto \int_{\Mb_{g,n}} \bar{\alpha}\cdot\bar{\beta}\cdot\lambda_g.
\end{equation}
    Here $\bar{\alpha}$ and $\bar{\beta}$ are arbitrary lifts of $\alpha$ and $\beta$ to $\Mb_{g,n}$. The latter pairing is well-defined because $\lambda_g$ vanishes on the boundary $\Mb_{g,n}\smallsetminus \M_{g,n}^{\ct}$ \cite[Equation 5]{FPGorenstein}. The codomain of the pairing is called the \emph{socle}. For the fact that the socle is one dimensional, see \cite{FaberPandharipande,GraberVakil}. If the pairings \eqref{barpairing} (respectively, \eqref{ctpairing}) are perfect, then we say $\R^*(\Mb_{g,n})$ (respectively, $\R^*(\M^{\ct}_{g,n})$) is \emph{Gorenstein}. We will also consider the analogous pairings in the cohomological tautological rings $\RH^*(\Mb_{g,n})$ and $\RH^*(\M_{g,n}^{\ct})$.

    Just as in the case of $\M_g$, if the Gorenstein property holds, it gives an algorithm for determining the ideals of tautological relations $\mathsf{I}_{\A}^*(\Mb_{g,n})$ and $\mathsf{I}_{\A}^*(\M_{g,n}^{\ct})$, thereby completely determining the structure of $\R^*(\Mb_{g,n})$ and $\R^*(\M_{g,n}^{\ct})$. The pairings \eqref{barpairing} and \eqref{ctpairing} can be computed explicitly and have been implemented in the Sage package \texttt{admcycles} \cite{Faberalgorithm,admcycles}.

 Note that $\M_{0,n}^{\ct}=\Mb_{0,n}$ is compact and the tautological ring is equal to the entire Chow or cohomology ring by a result of Keel \cite{Keel}. Therefore, the tautological rings in genus $0$ are Gorenstein by Poincar\'e duality. Furthermore, Tavakol \cite{Tavakol} proved $\R^*(\M_{1,n}^{\ct})$ is always Gorenstein, and Petersen \cite{Petersengenus1} proved $\R^*(\Mb_{1,n})$ is always Gorenstein. 
 
 Nevertheless, neither $\R^*(\Mb_{g,n})$ nor $\R^*(\M^{\ct}_{g,n})$ is always Gorenstein. Petersen and Tommasi proved $\R^*(\Mb_{2,n})$ is not Gorenstein when $n\geq 20$ \cite{PetersenTommasi}, and recent work of the first named author shows that $\R^*(\Mb_{g,n})$ is not Gorenstein for $g\geq 2$ and $2g+n\geq 24$ \cite{CGorenstein}. For compact type moduli, Petersen \cite{Petersen} proved the tautological ring of $\M^{\ct}_{2,n}$ is not Gorenstein when $n\geq 8$.
Our first theorem extends Petersen's compact type result to higher genera.
\begin{thm}\label{Gortheorem}
    If $g\geq 2$ and $2g+n\geq 12$, the tautological rings $\R^*(\M_{g,n}^{\ct})$ and $\R\H^*(\M_{g,n}^{\ct})$ are not Gorenstein.
\end{thm}
The proof of Theorem \ref{Gortheorem} goes by reducing to the case $g\geq 2$ and $2g+n=12$, and studying each such pair $(g,n)$ individually. It is not clear what the relationship between the failure of the Gorenstein conjecture is for the cases when $2g+n=12$, as there is no gluing map between the corresponding moduli spaces. The analogous result for $\Mb_{g,n}$ is proven by reduction to the case $g=2$ and $n=20$, using the self-gluing map \cite{PetersenTommasi,CGorenstein}.

We obtain a partial converse to Theorem \ref{Gortheorem}. 
\begin{thm}\label{isGorenstein}
    If $g=0,1$ or $g\geq 2$, $2g+n<12$, and $(g,n)\neq (2,7)$, then $\R^*(\M_{g,n}^{\ct})$ and $\RH^*(\M_{g,n}^{\ct})$ are isomorphic and Gorenstein.
\end{thm}
As noted above, the cases $g=0,1$ are due to Keel \cite{Keel} and Tavakol \cite{Tavakol}, respectively. The proof of Theorem \ref{isGorenstein} uses Theorem \ref{socleminus1} below. We expect that the full converse of Theorem \ref{Gortheorem} holds. In the case $(g,n)=(2,7)$, computer calculations have shown that the pairing is perfect in all degrees except the middle $\R^4(\M_{2,7}^{\ct}) \times \R^4(\M_{2,7}^{\ct}) \to \qq$.
See Section \ref{computationalaspects} for a discussion of the computational aspects. In Table \ref{Gortable} below, we record the ranks of the tautological groups when $g\geq 2$ and $2g+n<12$ except for $(g,n)=(2,7)$ in codimension $4$, where we have only managed to obtain an upper bound.

\begin{center}
\begin{table}[h]
\begin{tabular}{ p{1cm}||p{1cm}|p{1cm}|p{1cm}|p{1cm}|p{1.5cm}|p{1cm}|p{1.2cm}|p{1cm}|p{1cm}|} 
 $(g,n)$ & 0 & 1 & 2 & 3 & 4 & 5 & 6 & 7 & 8  \\ 
 \hline
 $(2,0)$ & 1 & 1 &  &  &  &  &  &  &    \\
 $(2,1)$ & 1 & 2 & 1 &  &  &  & &  &    \\
 $(2,2)$ & 1 & 5 & 5 & 1 &  &  &  &  &    \\
 $(2,3)$ & 1 & 11 & 24 & 11 & 1 &  &  &  &    \\
 $(2,4)$ & 1 & 23 & 101 & 101 & 23 & 1 &  &  &    \\ 
 $(2,5)$ & 1 & 47 & 384 & 769 & 384 & 47 & 1 &  &    \\
 $(2,6)$ & 1 & 95 & 1362 & 4981 & 4981 & 1362 & 95 & 1 &    \\
 $(2,7)$ & 1 & 191 & 4610 & 28606  & $\leq$ 52330 & 28606 & 4610 & 191 & 1   \\
 \hline
 $(3,0)$ & 1 & 2 & 2 & 1 &  & &  &  &    \\
 $(3,1)$ & 1 & 4 & 7 & 4 & 1 &  &  &  &    \\
 $(3,2)$ & 1 & 8 & 24 & 24 & 8 & 1 &  &  &    \\
 $(3,3)$ & 1 & 16 & 82 & 144 & 82 & 16 & 1 &  &    \\
 $(3,4)$ & 1 & 32 & 274 & 813 & 813 & 274 & 32 & 1 &    \\
 $(3,5)$ & 1 & 64 & 895 & 4281 & 7258 & 4281 & 895 & 64  & 1   \\
 \hline
 $(4,0)$ & 1 & 3 & 6 & 6 & 3 & 1 &  &  &    \\
 $(4,1)$ & 1 & 5 & 17 & 25 & 17 & 5 & 1 &  &    \\
 $(4,2)$ & 1 & 10 & 51 & 120 & 120 & 51 & 10 & 1 &    \\
 $(4,3)$ & 1 & 20 & 158 & 568 & 882 & 568 & 158 & 20 & 1   \\
 \hline
  $(5,0)$ & 1 & 3 & 10 & 19 & 19 & 10 & 3 & 1 &    \\
 $(5,1)$ & 1 & 6 & 28 & 75 & 107 & 75 & 28 & 6 & 1  \\
\end{tabular}

\caption{The ranks of $\R^i(\M_{g,n}^{\ct})\cong \RH^{2i}(\M_{g,n}^{\ct})$ when $g\geq 2$ and $2g+n<12$.}\label{Gortable}
\end{table}
\end{center}

The proof of Theorem \ref{Gortheorem} gives concrete examples of nonzero classes in $\R^i(\M_{g,n}^{\ct})$ that pair to zero with every class in $\R^{2g-3+n-i}(\M_{g,n}^{\ct})$ for some $i\geq 5$. We conjecture that the codimension $5$ classes are of the minimal possible codimension.
\begin{conj}\label{lowiconjecture}
    For $i\leq 3$, the pairings
    \[
    \R^i(\M_{g,n}^{\ct})\times \R^{2g-3+n-i}(\M_{g,n}^{\ct})\rightarrow \R^{2g-3+n}(\M_{g,n}^{\ct})
    \]
   % and
    %\[
    %\RH^{2i}(\M_{g,n}^{\ct})\times \RH^{4g-6+2n-2i}(\M_{g,n}^{\ct})\rightarrow \RH^{4g-6+2n}(\M_{g,n}^{\ct})
    %\]
    are perfect. If $i=4$, the pairing is of rank $\dim \R^4(\M_{g,n}^{\ct})$. %and is perfect if $2g-7+n\neq 5$.
\end{conj}

%\hannah{Somehow, I think this conjecture would be more natural without the ``and is perfect if $2g-7+n\neq 5$" because then I start to wonder why we should believe that.
%I guess $\M_7^{\ct}$ is some evidence in that direction -- based just on that though, are you actually thinking more boldly that the minimal $i$ such that a Gorenstein kernel exists in $\R^{2g - 3 + n - i}$ is increasing with the genus? If we leave off the last part, then the conjecture seems simpler: it's saying we think the minimal $i$ such that a Gorenstein failure occurs in $\R^{2g - 3 + n - i}$ is $i = 4$ and because we think the rank of the pairing is $\dim \R^4$, in order for a failure to possibly happen, $2g - 3 + n - 4 > 4$, i.e. $2g + n \geq 12$.}

Conjecture \ref{lowiconjecture} implies the analogous conjecture in cohomology (see Proposition \ref{chowvscoh}).
When $i=0$, the conjecture holds, simply because the socle $\R^{2g-3+n}(\M_{g,n}^{\ct})$ is one dimensional \cite{FaberPandharipande,GraberVakil}. We prove the $i=1$ case of Conjecture \ref{lowiconjecture}.
\begin{thm}\label{socleminus1}
    The pairings
    \[
    \R^1(\M_{g,n}^{\ct})\times \R^{2g-4+n}(\M_{g,n}^{\ct})\rightarrow \R^{2g-3+n}(\M_{g,n}^{\ct})
    \]
    %and
    %\[
    %\RH^2(\M_{g,n}^{\ct})\times \R^{4g-8+2n}(\M_{g,n}^{\ct})\rightarrow \R^{4g-6+n}(\M_{g,n}^{\ct})
    %\]
    are perfect.
\end{thm}
%\noindent %Theorem \ref{socleminus1} is used in the proof of Theorem \ref{Gortheorem}.
%The analogous pairing in cohomology is also perfect, see Proposition \ref{chowvscoh}. 

\subsection{Pixton's conjecture}\label{pixconj}
Pixton gave an alternate proposal for the structure of the tautological ring \cite{Pixton}. He defined a subspace $\FZ^*(\Mb_{g,n})\subset \S^*(\Mb_{g,n})$ that he conjectured was contained in $\mathsf{I}_{\A}^*(\Mb_{g,n})$, and hence $\mathsf{I}_{\H}^*(\Mb_{g,n})$.\footnote{$\mathsf{FZ}$ is an extension of the Faber--Zagier relations from $\M_g$ to $\Mb_{g,n}$.} The inclusion $\FZ^*(\Mb_{g,n})\subset \mathsf{I}_{\H}^*(\Mb_{g,n})$ was proven by Pandharipande--Pixton--Zvonkine \cite{PandharipandePixtonZvonkine} using the $3$-spin cohomological field theory. Later, Janda proved the inclusion $\FZ^*(\Mb_{g,n})\subset \mathsf{I}_{\A}^*(\Mb_{g,n})$. We now call the space $\FZ^*(\Mb_{g,n})$ the \emph{$3$-spin relations}.\footnote{The $3$-spin relations are also known as the \emph{generalized Faber--Zagier relations} and \emph{Pixton's relations}.} By restriction, one also obtains relations on $\M_{g,n}^{\ct}$, denoted $\FZ^*(\M_{g,n}^{\ct})$. We write
\[
\R^*_{\FZ}(\Mb_{g,n})=\S^*(\Mb_{g,n})/\FZ^*(\Mb_{g,n}) \quad \text{ and } \quad \R^*_{\FZ}(\M_{g,n}^{\ct})=\S^*(\M^{\ct}_{g,n})/\FZ^*(\M^{\ct}_{g,n}).
\]
\begin{conj}[Pixton \cite{Pixton}]\label{Pixton}
The $3$-spin relations are complete in Chow and cohomology:
\[
\R^*_{\FZ}(\Mb_{g,n}) = \R^*(\Mb_{g,n}) = \RH^*(\Mb_{g,n}) \quad \text{and} \quad \R^*_{\FZ}(\M_{g,n}^{\ct}) = \R^*(\M_{g,n}^{\ct})=\RH^*(\M_{g,n}^{\ct}).
%\FZ^*(\Mb_{g,n})=\mathsf{I}_{\A}^*(\Mb_{g,n})=\mathsf{I}_{\H}^*(\Mb_{g,n}), \qquad \FZ^*(\M^{\ct}_{g,n})=\mathsf{I}_{\A}^*(\M^{\ct}_{g,n})=\mathsf{I}_{\H}^*(\M^{\ct}_{g,n}).
\]
\end{conj}
There has been significant effort to produce relations in the tautological ring, but the only known relations are contained in the span of the $3$-spin relations \cite{CladerJanda,Jandarelations}, and so Conjecture \ref{Pixton} remains open. Pixton's conjecture in codimension $0$ is trivial. Gubarevich \cite{Gubarevich} proved Pixton's conjecture in codimension $1$ for $\Mb_{g,n}$:
\[
\R^1_{\FZ}(\Mb_{g,n}) = \R^1(\Mb_{g,n}) = \RH^2(\Mb_{g,n}).
\]
Kramer, Labib, Lewanski, and Shadrin \cite{KLLS} showed that the $3$-spin relations imply Graber and Vakil's Theorem $\star$ \cite{GraberVakil}. Arguing as in \cite[Section 5.5]{GraberVakil}, one sees that Pixton's conjecture holds for $\R^{3g-3+n}(\Mb_{g,n})=\RH^{6g-6+2n}(\Mb_{g,n})$.
Moreover, using the same line of reasoning, the thesis of Al-Aidroos \cite{AA} establishes Pixton's conjecture in dimension $1$ and $2$: $\R^{d - i}_{\FZ}(\Mb_{g,n}) = \R^{d - i}(\Mb_{g,n}) = \R\H^{2d - 2i}(\Mb_{g,n})$ for $i = 1, 2$ and $d = \dim \Mb_{g,n}$.

For compact type moduli, we prove some new cases of the conjecture.
\begin{thm}\label{Pixtonminus1}
    For $i\in \{0, 1, 2g-4+n, 2g-3+n\}$,
    \[
    \R^i_{\FZ}(\M^{\ct}_{g,n}) = \R^i(\M^{\ct}_{g,n}) = \RH^{2i}(\M^{\ct}_{g,n}).
    \]
\end{thm}
The $i=0$ case is trivial, and the $i=1$ case follows quickly from Gubarevich's result (see Corollary \ref{ctdivisors}). For $i=2g-3+n$ and $2g-4+n$, the proof uses that the $3$-spin relations imply Graber and Vakil's Theorem $\star$ \cite{KLLS,GraberVakil}.

%We also show that the $3$-spin relations are complete in the cases when the tautological ring is Gorenstein.
%\begin{thm}\label{PixwhenGor}
 %If $g=0,1$ or $g\geq 2$, $2g+n<12$, and $(g,n)\neq (2,7)$ or $(3,5)$, then 
 %\[
% \R^*_{\FZ}(\M^{\ct}_{g,n}) = \R^*(\M^{\ct}_{g,n}) = \RH^{*}(\M^{\ct}_{g,n}).
% \]
%\end{thm}
%For small $g$ and $n$, it is often the case that 

%\begin{thm}\label{PixtonwhenGorenstein}
%        If $g=0,1$ or $2g+n<12$ and $(g,n)\neq (2,7), (3,5)$, then the $3$-spin relations for $\R^*(\M_{g,n}^{\ct})$ and $\RH^*(\M_{g,n}^{\ct})$ are complete.
%\end{thm}
%Because the $3$-spin relations restrict the the Faber--Zagier relations on $\M_g$, Conjecture \ref{Pixton} predicts that $\R^*(\M_{24})$ is not Gorenstein. 
Using computer calculations, Pixton showed Conjecture \ref{Pixton} implies that $\R^*(\M_6^{\ct})$ and $\R^*(\M_{5,2}^{\ct})$ are not Gorenstein \cite{Pixton}, which is now confirmed by Theorem \ref{Gortheorem}. As the number of marked points increases, the computations become significantly more difficult. %In particular, we do not know what Pixton's conjecture predicts for $\R^*(\M_{2,8}^{\ct})$. 
The next theorem provides the first cases where Pixton's conjecture holds, but the tautological ring is not Gorenstein.

\begin{thm}\label{Pixtontheorem}
    The $3$-spin relations are complete in  Chow and cohomology for $\M_{6}^{\ct}$, $\M_{5,2}^{\ct}$, and $\M_7^{\ct}$. In each case, the Gorenstein kernel is $1$-dimensional and occurs in degree $\lceil \frac{2g-3+n}{2} \rceil$. %The $3$-spin relations are complete for the $\ss_2\times\ss_2\times\ss_2$ invariant Chow and cohomology of $\M_{3,6}^{\ct}$.
\end{thm}
In Table \ref{Pixtontheoremtable} below, we record the ranks of the tautological groups for the moduli spaces in Theorem \ref{Pixtontheorem}. 

\begin{center}
\begin{table}[h]
\begin{tabular}{p{0.85cm}||p{0.85cm}|p{0.85cm}|p{0.85cm}|p{0.85cm}|p{0.85cm}|p{0.85cm}|p{0.85cm}|p{0.85cm}|p{0.85cm}|p{0.85cm}|p{0.85cm}|p{0.85cm}|} 
 $(g,n)$ & 0 & 1 & 2 & 3 & 4 & 5 & 6 & 7 & 8 & 9 & 10 & 11\\ 
 \hline
 $(5,2)$ & 1 & 12 & 82 & 314 & \textbf{636} & \textbf{637} & 314 & 82 & 12 & 1 & & \\
 \hline
 $(6,0)$ & 1 & 4 & 15 & 42 & \textbf{71} & \textbf{72} & 42 & 15 & 4 & 1 & & \\
 \hline
  $(7,0)$ & 1 & 4 & 20 & 69 & 171 & \textbf{277} & \textbf{278} & 171 & 69 & 20 & 4 & 1  \\
\end{tabular}

\caption{The ranks of $\R^i(\M_{g,n}^{\ct})\cong \RH^{2i}(\M_{g,n}^{\ct})$ for $(g,n)$ as in Theorem \ref{Pixtontheorem}, with the ranks breaking the symmetry marked in bold.}\label{Pixtontheoremtable}
\end{table}
\end{center}
\subsection{The kernel of the pairing and moduli of abelian varieties}
Let $\mathcal{A}_g$ denote the moduli space of principally polarized abelian $g$-folds, and let $p:\X_g\rightarrow \mathcal{A}_g$ be the universal abelian variety.  The lambda classes on $\mathcal{A}_g$ are defined as $\lambda_i = c_i(p_*\Omega_p)$, and the tautological ring $\R^*(\mathcal{A}_g)$ is the subring of the Chow ring of $\mathcal{A}_g$ generated by the lambda classes. 

The Torelli map $
\mathsf{Tor}:\M_g^{\ct}\rightarrow \mathcal{A}_g$
associates to a curve of compact type its Jacobian, which is the product of the Jacobians of its components. %We have 
%\[\mathsf{Tor}^*\lambda_i=\lambda_i\in \R^i(\M_g^{\ct}).\] 
The first interesting example of a non-tautological class in $\A^*(\mathcal{A}_g)$ was found recently \cite[Theorem 5]{COP}: for $g=6$, we have 
$$[\mathcal{A}_1\times \mathcal{A}_{5}]\notin \R^*(\mathcal{A}_6)\,.$$ %The proof crucially uses Theorem \ref{Pixtontheorem}, and provides a geometric explanation for the failure of the Gorenstein property in Theorem \ref{Gortheorem} when $g=6$ and $n=0$. 
First, by \cite[Proposition 2]{COP}, if $[\mathcal{A}_1\times \mathcal{A}_{g-1}]\in \R^*(\mathcal{A}_g)$, then
\[
[\mathcal{A}_1\times \mathcal{A}_{g-1}] = \frac{g}{6|B_{2g}|}\lambda_{g-1}.
\]
Therefore, if the class
\[
\Delta_g = [\mathcal{A}_1\times \mathcal{A}_{g-1}] - \frac{g}{6|B_{2g}|}\lambda_{g-1}
\]
is nonzero, then $[\mathcal{A}_1\times \mathcal{A}_{g-1}]$ is not tautological. If $\mathsf{Tor}^*\Delta_g\neq 0$, then $\Delta_g\neq 0$. However, by \cite[Theorem 4]{COP} the pullback $\mathsf{Tor}^*\Delta_g\in \R^{g-1}(\M_g^{\ct})$ lies in the kernel of the pairing \eqref{ctpairing}. Therefore, it is difficult to test whether $\mathsf{Tor}^*\Delta_g$ is nonzero. Theorem \ref{Pixtontheorem} provides a complete description of $\R^*(\M_6^{\ct})$, and is therefore an essential ingredient in the proof of \cite[Theorem 5]{COP}.
Using Theorem \ref{Pixtontheorem}, it is shown in \cite[Theorem 5]{COP} that $\mathsf{Tor}^*\Delta_6$ generates the $1$-dimensional kernel of the pairing 
\[
\R^{4}(\M_{6}^{\ct})\times \R^5(\M_6^{\ct})\rightarrow \qq.
\]
Surprisingly, $\mathsf{Tor}^*\Delta_7=0$ \cite[Proposition 7]{COP}, but it is suspected that $\mathsf{Tor}^*\Delta_g\neq 0$ for $g\geq 8$.

The fact that $[\mathcal{A}_1\times \mathcal{A}_5]$ is not tautological provides a geometric explanation for the failure of the Gorenstein property for $\R^*(\M_{6}^{\ct})$. An analogous geometric explanation for the failure of the Gorenstein property of $\R^*(\M_{g,n}^{\ct})$ when $g\geq 2$, $n>0$, and $2g+n=12$ has been pursued in \cite{10author}, where again Theorem \ref{Pixtontheorem} is a crucial ingredient.

\subsection*{Plan of the paper}
Sections \ref{yesgorensteinsection}, \ref{3spinbarct}, and \ref{socleminus1section} deal with the cases when the tautological ring is Gorenstein. In Section \ref{yesgorensteinsection}, we prove Theorem \ref{isGorenstein}. In Section \ref{3spinbarct}, we discuss the relationship between Pixton's Conjecture \ref{Pixton} for stable and compact type moduli. As a quick application, we prove the $i=1$ case of Theorem \ref{Pixtonminus1}. In Section \ref{socleminus1section}, we prove Theorems \ref{socleminus1} and \ref{Pixtonminus1} simultaneously. 

Sections \ref{prelim} and \ref{2gn12} deal with the failure of the Gorenstein property. In Section \ref{prelim}, we reduce Theorem \ref{Gortheorem} to the cases $g\geq 2$ and $2g+n=12$. In Section \ref{2gn12}, we study these cases, proving Theorem \ref{Pixtontheorem} and finishing the proof of Theorem \ref{Gortheorem}. 

Finally, in Section \ref{computationalaspects}, we give a more detailed description of the computer computations used throughout the paper.

\subsection*{Acknowledgments}
We thank Jonas Bergstr\"om, Carel Faber, Dragos Oprea, Rahul Pandharipande, Dan Petersen, Aaron Pixton for helpful conversations. We are grateful to Carel Faber for sharing his point counting data in genus $4$.
We also thank Charles Bouillaguet for advice on using the \texttt{SpaSM} library \cite{spasm} for Gaussian elimination modulo $p$.

Many of the computer checks in this paper were carried out on the servers of ETH Z\"urich and UZH Z\"urich. We thank the respective IT support groups for their help in facilitating these calculations.

\section{When the tautological ring is Gorenstein}\label{yesgorensteinsection}
In this section, assuming Theorem \ref{socleminus1}, we prove Theorem \ref{isGorenstein}. Theorem \ref{socleminus1} will be proven in Section \ref{socleminus1section}. 

The socle $\R^{2g-3+n}(\M_{g,n}^{\ct})$ is generated by the fundamental class of the locus parametrizing maximally degenerate $n+g$ pointed rational curves with $g$ elliptic tails attached \cite[Section 4.1.2]{FaberPandharipande}. That is, the socle is generated by the image of the point class on $\M^{\ct}_{0,n+g}$ under the composition
\[\R^{g+n-3}(\M^{\ct}_{0,n+g}) \to \R^{g+n-3}(\M^{\ct}_{0,n+g} \times (\M_{1,1}^{\ct})^{\times g}) \to \R^{2g + n - 3}(\M_{g,n}^{\ct}),\]
where the first map is the pullback along the projection and the second map is the pushforward along the gluing map.
This description of the socle holds in both Chow and cohomology, so the cycle class map in the socle degree
\begin{equation}\label{socledegreeiso}
c:\R^{2g-3+n}(\M_{g,n}^{\ct})\xrightarrow{\cong} \R\H^{4g-6+2n}(\M_{g,n}^{\ct})
\end{equation}
is an isomorphism. Using the isomorphism \eqref{socledegreeiso}, we see that the Gorenstein property in Chow and cohomology are closely related.

\begin{prop}\label{chowvscoh}
    If the pairing 
    \[
    \R^i(\M_{g,n}^{\ct})\times \R^{2g-3+n-i}(\M_{g,n}^{\ct})\rightarrow \qq
    \]
    is perfect, then so is the pairing
    \[
    \R\H^{2i}(\M_{g,n}^{\ct})\times \RH^{4g-6+2n-2i}(\M_{g,n}^{\ct})\rightarrow \qq.
    \]
\end{prop}
\begin{proof}
    The proof is entirely analogous to \cite[Corollary 2.5]{PetersenTommasi}. To obtain a contradiction, assume the latter pairing is not perfect. Then the cycle class map $c:\R^*(\M_{g,n}^{\ct})\rightarrow \R\H^{2*}(\M_{g,n}^{\ct})$ is not an isomorphism in degree $i$ or $2g-3+n-i$. It is by definition surjective, so there must be an element $\alpha$ such that $c(\alpha)=0$. By assumption, there is an element $\beta\in \R^*(\M_{g,n}^{\ct})$ of complementary degree such that $\alpha\cdot \beta\neq 0$. Because $\alpha\cdot \beta$ is in the socle degree,
    \[
    0\neq c(\alpha\cdot\beta)=c(\alpha)\cdot c(\beta),
    \]
    contradicting $c(\alpha)=0$.
\end{proof}

\begin{proof}[Proof of Theorem \ref{isGorenstein}, assuming Theorem \ref{socleminus1}]
By Proposition \ref{chowvscoh}, we can work in the tautological Chow ring $\R^*(\M_{g,n}^{\ct})$. We may also assume $g\geq 2$, as Theorem \ref{isGorenstein} holds when $g=0,1$ by \cite{Keel,Tavakol}. 

Let $g\geq 2$, $2g+n<12$, and $(g,n)\neq (2,7)$. We need to show that the pairings
\[
\R^i(\M_{g,n}^{\ct})\times \R^{2g-3+n-i}(\M_{g,n}^{\ct})\rightarrow \qq
\]
are perfect. %for $0\leq i \leq \lfloor\frac{2g-3+n}{2}\rfloor$. 
For $i=0$, the pairing is perfect because we know $\R^{2g - 3 + n}(\M_{g,n}^{\ct}) \cong \qq$. For $i = 1$, the result follows from Theorem \ref{socleminus1}. We now assume $i \geq 2$.

Using \texttt{admcycles}, we compute a matrix $M^{\FZ,i}(\M_{g,n}^{\ct})$ whose columns are indexed by the decorated graph generators for $\S^i(\M_{g,n}^{\ct})$ and whose rows correspond to elements of a generating set for $\FZ^i(\M_{g,n}^{\ct})$. We reduce the entries of the matrix modulo $p$ for some prime $p$, obtaining a matrix $M^{\FZ,i}_{\mathbb{F}_p}(\M_{g,n}^{\ct})$. Here, we only use that the denominators of the entries in $M^{\FZ,i}(\M_{g,n}^{\ct})$ are not divisible by $p$. We have 
\[
\rank M^{\FZ,i}_{\mathbb{F}_p}(\M_{g,n}^{\ct}) \leq \rank M^{\FZ,i}(\M_{g,n}^{\ct}).
\]
%\hannah{could I also just think of scaling the relations so that the matrix has no denominators and then the above is a true statement for any $p$? (it just would have much less hope of being an equality if $p$ had divided the denominators)}
We compute $\rank M^{\FZ,i}_{\mathbb{F}_p}(\M_{g,n}^{\ct})$. Then
\[
\dim \S^i(\M_{g,n}^{\ct})-\rank M^{\FZ,i}_{\mathbb{F}_p}(\M_{g,n}^{\ct})
\]
is an upper bound for the dimension of $\R^i(\M_{g,n}^{\ct})$.\footnote{When $g=2$, we calculate the upper bound for the dimension of $\R^i(\M_{2,n}^{\ct})$ only when $i\leq \lfloor\frac{2g-3+n}{2}\rfloor$ because the tautological Betti numbers are known to be symmetric when $n\leq 7$ \cite[Theorem 3.6]{Petersen}.}

Next, we bound from below the ranks of the pairings
\[
\R^i(\M_{g,n}^{\ct})\times \R^{2g-3+n-i}(\M_{g,n}^{\ct})\rightarrow \qq,
\]
again using \texttt{admcycles}.
In each case, the upper bound and lower bound agree. The results are recorded in Table \ref{Gortable}. See Section \ref{computationalaspects} for more details on the computer implementation.
\end{proof}
\section{The \texorpdfstring{$3$}{3}-spin relations on \texorpdfstring{$\Mb_{g,n}$}{Mbar\_gn} and \texorpdfstring{$\M_{g,n}^{\ct}$}{M\_gn\^ct} }\label{3spinbarct}
In this section, we give a method for proving the $3$-spin relations are complete for $\M_{g,n}^{\ct}$. The method depends on the completeness of the $3$-spin relations for the cohomology of $\Mb_{g,n}$ and  $\Mb_{g-1,n+2}$.
\begin{lem}\label{tautexact}
Suppose that $\H^{2k-2}(\Mb_{g-1,n+2})=\R\H^{2k-2}(\Mb_{g-1,n+2})$. Then the sequence
\[
\R\H^{2k-2}(\Mb_{g-1,n+2})\rightarrow \R\H^{2k}(\Mb_{g,n})\rightarrow \R\H^{2k}(\M_{g,n}^{\ct})\rightarrow 0
\]
is exact.
\end{lem}
\begin{proof}
By definition, the restriction map 
\[
\R\H^{2k}(\Mb_{g,n})\rightarrow \R\H^{2k}(\M_{g,n}^{\ct})
\]
is surjective, and the pushforward map $\H^{2k-2}(\Mb_{g-1,n+2})\rightarrow \H^{2k}(\Mb_{g,n})$ sends tautological classes to tautological classes. Therefore, we have a commutative diagram %\hannah{you want the upper left to be $\H^{2k-2}$ right?}
\[
\begin{tikzcd}
{\H^{2k-2}(\Mb_{g-1,n+2})} \arrow[equal]{d} \arrow[r] & {\H^{2k}(\Mb_{g,n})} \arrow[r]                  & {W_{2k} \H^{2k}(\M^{\ct}_{g,n})} \arrow[r]           & 0 \\
{\R\H^{2k-2}(\Mb_{g-1,n+2})} \arrow[r]                                & {\R\H^{2k}(\Mb_{g,n})} \arrow[u, hook] \arrow[r] & {\R\H^{2k}(\M^{\ct}_{g,n})} \arrow[u, hook] \arrow[r] & 0 \, ,
\end{tikzcd}
\]
where the top row, coming from the long exact sequence in cohomology, is exact. A diagram chase shows that the bottom row is exact as well.
\end{proof}

\begin{lem}\label{bartoct}
    Suppose that $\H^{2k-2}(\Mb_{g-1,n+2})=\R\H^{2k-2}(\Mb_{g-1,n+2})$. If the $3$-spin relations are complete for $\R\H^{2k}(\Mb_{g,n})$, then they are complete for $\R\H^{2k}(\M_{g,n}^{\ct})$.
\end{lem}
\begin{proof}
By Lemma \ref{tautexact}, the top row of the diagram below
\begin{center}
\begin{tikzcd}
\R\H^{2k-2}(\Mb_{g-1,n+2}) \arrow{r} & \R\H^{2k}(\Mb_{g,n}) \arrow{r} & \R\H^{2k}(\M_{g,n}^{\ct})\arrow{r} & 0 \\
\S^{k-1}(\Mb_{g-1,n+2}) \arrow[two heads]{u} \arrow{r} & \S^k(\Mb_{g,n}) \arrow[two heads]{u}
\end{tikzcd}
\end{center}
is exact. Hence, $\R\H^{2k}(\M_{g,n}^{\ct})=\R\H^{2k}(\Mb_{g,n})/\S^{k-1}(\Mb_{g-1,n+2})$. By assumption, $\R\H^{2k}(\Mb_{g,n})=\S^{k}(\Mb_{g,n})/\mathsf{FZ}^{k}(\Mb_{g,n})$. %where $\mathsf{FZ}$ denotes the ideal of $3$-spin relations. 
We have
\begin{equation*}
    \begin{split}
        \R\H^{2k}(\M_{g,n}^{\ct})&=\S^{k}(\Mb_{g,n})/(\mathsf{FZ}^{k}(\Mb_{g,n})+\S^{k-1}(\Mb_{g-1,n+2}))\\
        % &= \big(\S^{k}(\Mb_{g,n})/\S^{k-1}(\Mb_{g-1,n+2})\big) /\big(\mathsf{FZ}^{k}(\Mb_{g,n})/[\S^{k-1}(\Mb_{g-1,n+2})\cap \mathsf{FZ}^{k}(\Mb_{g,n})]\big)\\
        &= \big(\S^{k}(\Mb_{g,n})/\S^{k-1}(\Mb_{g-1,n+2})\big) /\big(\mathrm{im}(\mathsf{FZ}^{k}(\Mb_{g,n}))\big)\\
        &= \S^{k}(\M_{g,n}^{\ct})/\mathsf{FZ}^{k}(\M^{\ct}_{g,n}),
    \end{split}
\end{equation*}
showing the $3$-spin relations are complete. Here $\mathrm{im}(\mathsf{FZ}^{k}(\Mb_{g,n}))$ is the image of $\mathsf{FZ}^{k}(\Mb_{g,n})$ in the quotient ring $\S^{k}(\Mb_{g,n})/\S^{k-1}(\Mb_{g-1,n+2})$, which itself is canonically isomorphic to the space $\S^{k}(\M_{g,n}^{\ct})$ of decorated strata of the compact-type moduli space.
%\hannah{small comment: above, we are writing $\S^{k-1}$ as an abbreviation for the image of $S^{k-1}$ under various maps (which are probably not injective)}
\end{proof}
\begin{rem}
    One can also prove versions of Lemmas \ref{tautexact} and \ref{bartoct} in the Chow ring, using the excision exact sequence for Chow groups instead of the long exact sequence in cohomology.
\end{rem}
 Using Lemma \ref{bartoct}, we show the $3$-spin relations are complete for $\M_{1,n}^{\ct}$ and that the $3$-spin relations are complete in codimension $1$.
\begin{cor}\label{genus13spin}
    The $3$-spin relations are complete in Chow and cohomology for $\M_{0,n}^{\ct}$ and $\M_{1,n}^{\ct}$.
\end{cor}
\begin{proof}
It suffices to prove the result in cohomology, as there can be no more relations in Chow than in cohomology.
The genus $0$ case is well-known. Indeed, by \cite{Keel}, the ideal $\mathsf{I}_{\H}(\Mb_{0,n})$ is generated by the WDVV relations, which are contained in $\FZ^*(\Mb_{0,n})$ \cite[Section 3.6]{PandharipandePixtonZvonkine}.

    When $g=1$, we use Lemma \ref{bartoct}. By \cite{Keel}, we have the equality $\H^*(\Mb_{0,n})=\RH^*(\Mb_{0,n})$. Moreover, the ideal of relations $\mathsf{I}_{\H}^*(\Mb_{1,n})$ is generated by the WDVV and Getzler relations \cite{Petersengenus1}. Both of these relations are known to be contained in $\FZ^*(\Mb_{1,n})$, see \cite[Section 3.6]{PandharipandePixtonZvonkine} and \cite[Section 4.6]{calculus} or \cite[p. 87]{Pixton}. Applying Lemma \ref{bartoct} yields the statement.
\end{proof}

\begin{cor}\label{ctdivisors}
    The $3$-spin relations are complete for $\R^1(\M_{g,n}^{\ct})$ and $\RH^2(\M_{g,n}^{\ct})$.
\end{cor}
\begin{proof}
    By \cite{Gubarevich}, the $3$-spin relations are complete for $\RH^2(\Mb_{g,n})$. The fundamental class $[\Mb_{g-1,n+2}]$ is tautological by definition, and so we may apply Lemma \ref{bartoct}.
\end{proof}

\section{Proof of Theorems \ref{socleminus1} and \ref{Pixtonminus1}}\label{socleminus1section}
\subsection{Overview}
In this section, we prove Theorems \ref{socleminus1} and \ref{Pixtonminus1}. A key tool is the following result.
\begin{thm}\label{theoremstar}
    Any $\alpha \in \R^d_{\FZ}(\Mb_{g,n})$ is a linear combination of decorated strata classes $[\Gamma,\gamma]$ such that $\Gamma$ has at least $d-g+1$ genus $0$ vertices.
\end{thm}
Theorem \ref{theoremstar} is often called Theorem $\star$, and was proven in $\R^d(\Mb_{g,n})$ by Graber and Vakil \cite{GraberVakil}. The stronger statement that Theorem $\star$ holds in $\R^d_{\FZ}(\Mb_{g,n})$ was proven recently by Kramer, Labib, Lewanski, and Shadrin \cite[Proposition 5.7 and Corollary 5.9]{KLLS}. Theorem \ref{Pixtonminus1} for $i=2g-3+n$ follows quickly from Theorem \ref{theoremstar} (see Section \ref{2g3n}).

Theorem \ref{Pixtonminus1} for $i=2g-4+n$ and Theorem \ref{socleminus1} will be proven simultaneously. We will show that the pairing
\begin{equation}  \label{eqn:R1FZ_pairing_is_perfect}
\R^1_{\FZ}(\M_{g,n}^{\ct})\times \R_{\FZ}^{2g-4+n}(\M_{g,n}^{\ct})\rightarrow \R^{2g-3+n}_{\FZ}(\M_{g,n}^{\ct})\cong \qq
\end{equation}
is perfect. The perfectness of the pairing forbids further relations, proving Theorem \ref{Pixtonminus1}. Therefore, the pairing 
\[
\R^1(\M_{g,n}^{\ct})\times \R^{2g-4+n}(\M_{g,n}^{\ct})\rightarrow \qq
\]
is perfect, proving Theorem \ref{socleminus1}.
\subsection{Proof of Theorem \ref{Pixtonminus1} for \texorpdfstring{$i=2g-3+n$}{i=2g-3+n}}\label{2g3n}
We follow \cite[Section 5.6]{GraberVakil}. For $(g,n)=(2,0)$ we have $2g-3+n=1$, so the result follows from Corollary \ref{ctdivisors}. For $(g,n)\neq (2,0)$, any stable graph $\Gamma$ without loops has at most $g-2+n$ genus $0$ vertices. By Theorem \ref{theoremstar}, any generator of $\R_{\FZ}^{2g-3+n}(\M_{g,n}^{\ct})$ thus has exactly $g-2+n$ genus $0$ vertices. Moreover, each genus $0$ vertex must be trivalent, and all other vertices are genus $1$ leaves. There can be no $\kappa$ or $\psi$ decorations. Repeatedly applying the WDVV relation on $\Mb_{0,4}$, which is a $3$-spin relation, we see any two such strata are equivalent in $\R^{2g-3+n}_{\FZ}(\M_{g,n}^{\ct})$. 
\qed
%that the pairings involving codimension $1$ classes are perfect.
%The reader who is only interested in the \emph{failure} of the Gorenstein conjecture does not need to read this section.
\subsection{Proof of Theorem \ref{Pixtonminus1} for \texorpdfstring{$i=2g-4+n$}{i=2g-4+n} and Proof of Theorem \ref{socleminus1}}
We prove that the pairing \eqref{eqn:R1FZ_pairing_is_perfect} is perfect by induction on $g$, using the map $\varphi: \M_{g-1,n+1}^{\ct} \times \M^{\ct}_{1,1} \to \M_{g,n}^{\ct}$, gluing an elliptic tail to the last marked point. The base cases for this induction are that the pairing
\[\R_{\FZ}^{1}(\M_{g,n}^{\ct}) \times \R_{\FZ}^{2g+n-4}(\M_{g,n}^{\ct}) \to \qq\]
is perfect when $g = 0$ \cite{Keel} and when $g = 1$ \cite{Tavakol} and the $3$-spin relations are complete when $g=0,1$ by Corollary \ref{genus13spin}. Note that by the geometric description of the socle as in Section \ref{yesgorensteinsection}, the map
\[
\varphi_*:\R_{\FZ}^{2g - 4 + n}(\M_{g-1,n+1}^{\ct}) \cong \R_{\FZ}^{2g - 4 + n}(\M_{g-1,n+1}^{\ct} \times \M^{\ct}_{1,1}) \to \R_{\FZ}^{2g - 3 + n}(\M_{g,n}^{\ct})
\]
is an isomorphism. 

Combined with the push-pull formula, this implies the following fact: for $w \in \R_{\FZ}^i(\M_{g,n}^{\ct})$ and
$\tilde{v} \in \R_{\FZ}^{2g - 4 + n - i}(\M_{g-1,n+1}^{\ct})$, we have
\begin{equation} \label{pushpull} w \cdot \varphi_* \tilde{v} = 0 \qquad \Longleftrightarrow \qquad \varphi^*w \cdot \tilde{v} = 0. 
\end{equation}
The proof that \eqref{eqn:R1FZ_pairing_is_perfect} is perfect breaks into two parts, given by Propositions \ref{R1part} and \ref{Rsminus1part} below.% The first is the following proposition.
\begin{prop}\label{R1part}
    For any nonzero element $w \in \R^1(\M_{g,n}^{\ct})=\R_{\FZ}^1(\M_{g,n}^{\ct})$, there exists some $v\in \R_{\FZ}^{2g-4+n}(\M_{g,n}^{\ct})$ such that $v\cdot w\neq 0$.
\end{prop}
\noindent Using the following lemma, the equivalence \eqref{pushpull}, and induction, we will reduce Proposition \ref{R1part} to $g=2$. The proof of Proposition \ref{R1part} in genus $2$ is similar to the $g\geq 3$ cases, but more technical, so we defer it until later.
\begin{lem} \label{pullbackinjective}
If $g \geq 3$, then 
\[\varphi^*: \R^1(\M_{g,n}^{\ct}) \to \R^1(\M_{g-1,n+1}^{\ct} \times \M^{\ct}_{1,1}) \cong \R^1(\M_{g-1,n+1}^{\ct})\]
is injective.    
\end{lem}
\begin{proof}

First suppose that $g \geq 4$ so that all separating boundary, $\psi$ and $\kappa$ classes are independent in $\R^1(\M_{g-1,n+1}^{\ct})$ by \cite[Theorem 2.2]{ArbarelloCornalba}. Let $\delta_{a,A} = [\Gamma_{a,A},1]$ denote the fundamental class of the boundary divisor associated to the graph $\Gamma_{a,A}$, consisting of  a vertex of genus $a$ hosting markings $A \subset \{1, \ldots,n\}$ with a single edge to a vertex of genus $g - a$ hosting markings $A^c$.
In order to avoid overcounting, we assume $a \leq g/2$ and, when $n > 0$, that $1 \in A$ if $a = g/2$.
The generic form of an element in $\R^1(\M_{g,n}^{\ct})$ is
\[w = \sum_{i=1}^n c_i \psi_i + \sum_{\substack{a \leq g/2 \\ 1 \in A \text{ if $a = g/2$} \\ (a, A) \neq (1, \emptyset)}} c_{a,A} \delta_{a, A}+ e \delta_{1, \emptyset}+ f \kappa_1 \in \R^1(\M_{g,n}^{\ct}) \]
for $c_i, c_{a,A}, e, f \in \qq$. Let $\{1, \ldots, n, x\}$ be the marking set on $\M_{g-1,n+1}^{\ct}$, where $\varphi$ glues an elliptic tail in at the marking $x$.
Then
\[\varphi^*w = \sum_{i=1}^n c_i \psi_i+ \sum_{\substack{a \leq g/2 \\ 1 \in A \text{ if $a = g/2$} \\ (a, A) \neq (1, \emptyset)}} c_{a,A} (\delta_{a-1, A+x}+ \delta_{a, A})+ e (-\psi_x+ \delta_{1, \emptyset})+ f \kappa_1 \in \R^1(\M_{g-1,n+1}^{\ct}). \]
Above, we use the convention $\delta_{-1,A+x} = 0$ to avoid writing out separate cases for the $a = 0$ terms in the sum.
We claim the classes appearing in the linear combination above are independent. From this it follows that if $\varphi^* w = 0$, then we have $c_i = c_{a,A}= e= f = 0$, so that $w = 0$, proving injectivity.

To verify the claim, we need to know that as we vary $A$ we never make $\delta_{a, A} = \delta_{a' - 1, A'+x}$, so that there is no way for cancellation to occur in the above sum. In order for that to happen, we would need $g - 1 - a = a' - 1$ and $A^c = A'+ x$. That would mean $g = a+ a'$. Since $a, a' \leq g/2$ this can only happen when $a, a' = g/2$. But in this case, we have assumed that $1 \in A$ and $1 \in A'$, so we cannot have $A^c = A'+ x$.

When $g = 3$, we take $w$ as before and the expression for $\varphi^*w$ is valid. However, this time, the $\psi$ classes and boundary divisors are a basis for
$\R^1(\M_{2,n+1}^{\ct})$, and there is a relation that expresses $\kappa_1$ in terms of them. Nevertheless, since $-\psi_x$ and $\delta_{1, \emptyset}$ are independent and appear only in the term $e(-\psi_x + \delta_{1,\emptyset})$, it suffices to see that the coefficients of $\psi_x$ and $\delta_{1, \emptyset}$ in $\kappa_1$ are not negatives of each other. This relation is given in \cite[Theorem 2.2(b)]{ArbarelloCornalba}. In the notation there, $\delta_a   = \sum_{A} \delta_{a, A}$, and $\psi = \sum \psi_i + \psi_x$, so $\delta_{1, \emptyset}$ appears with coefficient $7/5$ in $\kappa_1$, while $\psi_x$ appears with coefficient $1$.
\end{proof}

\begin{proof}[Proof of Proposition \ref{R1part} for $g \geq 3$ assuming the case $g=2$]
  Let $w\in \R^1(\M_{g,n}^{\ct})$ be nonzero. By Lemma \ref{pullbackinjective}, $\varphi^* w\neq 0$. Thus, there exists a class $\tilde{v}\in \R_{\FZ}^{2g-5+n}(\M_{g-1,n+1}^{\ct})$ such that $\tilde{v}\cdot\varphi^*w\neq 0$ by induction on $g$. Setting $v=\varphi_*\tilde{v}$, we see $v\cdot w\neq 0$ by \eqref{pushpull}.
\end{proof}

The second direction in Theorem \ref{socleminus1} is the following.
\begin{prop}\label{Rsminus1part}
    For any nonzero element $v\in \R_{\FZ}^{2g-4+n}(\M_{g,n}^{\ct})$, there exists some class $w\in \R^1(\M_{g,n}^{\ct})$ such that $v\cdot w\neq 0$.

\end{prop}
The following two lemmas describe all such classes $v$. Lemma \ref{pushforwardsurjective} plays a role dual to Lemma \ref{pullbackinjective}. 
\begin{lem}\label{theoremstarconsequence}
    Any class $v \in \R^{2g-4+n}_{\FZ}(\M_{g,n}^{\ct})$ is a linear combination of decorated strata classes $[\Gamma,\gamma]$ with the following features:
    \begin{enumerate}
        \item There is one vertex of $\Gamma$ of type $(0,4)$, $(1,2)$, $(2,0)$ or $(2,1)$.
        \item All other vertices of $\Gamma$ are of type $(0,3)$ or $(1,1)$.
        \item If $\Gamma$ has no vertex of type $(2,1)$, then $\gamma$ is of degree $0$.
        \item If $\Gamma$ has a vertex of type $(2,1)$, then $\gamma$ is of degree $1$ on the type $(2,1)$ vertex and degree $0$ on all other vertices.
    \end{enumerate}
\end{lem}
\begin{proof}

We first prove parts (1) and (2). By Theorem \ref{theoremstar}, any nontrivial generator of the group $\R^{2g-4+n}_{\FZ}(\M_{g,n}^{\ct})$ has at least $g - 3 + n$ genus $0$ vertices, which is one less than the maximum possible number of genus $0$ vertices, $g - 2 + n$. 

We proceed by induction on $g$ and $n$. As $\Gamma$ is a tree, we may suppose it has a leaf of some type $(g',n')$. This leaf is glued to a tree $\Gamma_0$ of genus $g - g'$ with $n - n' + 2$ markings. First suppose $g' = 0$. Then $\Gamma_0$ must have at least $g - 4 + n$ genus $0$ vertices. Hence, we have $g - 4 + n \leq (g - g') - 2 + (n - n' +2)$, which implies $n' \leq 4$. That is, any genus $0$ leaf is type $(0, 4)$ or $(0, 3)$. If the leaf is type $(0, 4)$, then $\Gamma_0$ has the maximal number of genus $0$ components given its genus so all other vertices are type $(0, 3)$ or $(1, 1)$. If the leaf is type $(0, 3)$, then $\Gamma_0$ has one less than the maximal number of genus $0$ components and the claim follows by induction.
Now suppose $g' > 0$. Then $\Gamma_0$ must have at least $g - 3 + n$ genus $0$ vertices. Hence, we have $g - 3 + n \leq (g - g') - 2 + (n - n' + 2)$, which implies $g' +n' \leq 3$. Thus, the allowable $(g', n')$ are $(1, 1), (1, 2),(2, 1),$ and $(2, 0)$.  If the leaf is type $(1,2)$ or $(2, 1)$, then $\Gamma_0$ has the maximal number of genus $0$ components, so all other vertices are type $(0, 3)$ or $(1, 1)$. If the leaf is type $(1, 1)$, then $\Gamma_0$ has one less than the maximal or the maximal number of genus $0$ components, and thus has the claimed form by induction.

To prove parts (3) and (4), note that the number of edges in a graph $\Gamma$ satisfying (1) and (2) is $2g - 4 + n$, unless $\Gamma$ has a vertex of type $(2, 1)$, in which case there are $2g - 5 + n$ edges. In the latter case, there must be a degree $1$ decoration and the only place it gives a non-vanishing generator is on the $(2, 1)$ vertex.
\end{proof}

\begin{lem}\label{pushforwardsurjective}%%%\item 
If $g \geq 3$, the map
\[\varphi_*:\R_{\FZ}^{2g-5+n}(\M_{g-1,n+1}^{\ct}) \cong \R_{\FZ}^{2g-5+n}(\M_{g-1,n+1}^{\ct} \times \M^{\ct}_{1,1}) \to \R_{\FZ}^{2g-4+n}(\M_{g,n}^{\ct})\]
is surjective.
\end{lem}
\begin{proof}
If $g \geq 3$, any class in $\R^{2g-4+n}_{\FZ}(\M_{g,n}^{\ct})$ must be supported on a graph with a $(1, 1)$ vertex by Lemma \ref{theoremstarconsequence}. Thus, all classes in codimension $2g - 4 + n$ are pushed forward from the elliptic tail divisor.
\end{proof}

Now consider the commutative diagram obtained by attaching two elliptic tails:
\begin{equation} \label{beq}
\begin{tikzcd}
& \M_{g-1,n+1}^{\ct} \times \M_{1,1}^{\ct} \arrow{dl}[swap]{\varphi} \\
 \M_{g,n}^{\ct}  & & \M_{g-2,n+2}^{\ct}\times \M_{1,1}^{\ct} \times \M_{1,1}^{\ct}. \arrow{ul}[swap]{\alpha} \arrow{dl}{\beta}\\
 & \M_{g-1,n+1}^{\ct} \times \M_{1,1}^{\ct} \arrow{ul}{\varphi}
\end{tikzcd}
\end{equation}
There is a complex
\begin{equation} \label{2tails}
\begin{tikzcd} \R^1(\M_{g,n}^{\ct}) \arrow{r}{\varphi^*} & \R^1(\M_{g-1,n+1}^{\ct} \times \M_{1,1}^{\ct}) \arrow{r}{\alpha^* - \beta^*} & \R^1(\M_{g-2,n+2}^{\ct} \times \M_{1,1}^{\ct} \times \M_{1,1}^{\ct}).
\end{tikzcd}
\end{equation}

\begin{lem} \label{cx}
If $g\geq 3$, then the complex \eqref{2tails} is exact.
\end{lem}
%\hannah{changed $v$'s in the proof to $u$'s in order to reserve $v$ for a class in socle minus $1$}
\begin{proof}
Let $x$ be the last marking on $\M_{g-1,n+1}^{\ct}$. 
Suppose $u \in \ker(\alpha^* - \beta^*) \subset \R^1(\M_{g-1,n+1}^{\ct})$. We wish to produce a class $w \in \R^1(\M_{g,n}^{\ct})$ such that $\varphi^* w = u$.
We first treat the case $g - 1 \geq 3$.
In this case, we can write $u$ uniquely as a sum of compact type graphs with one edge and $\psi$ and $\kappa$ classes
\[u = \sum_{\Gamma} a_{\Gamma} [\Gamma] + \sum_{i = 1}^{n} c_i \psi_i + c_x \psi_x + d \kappa_1.\]
Given a graph $\Gamma$, write $g(x)$ for the genus of the vertex containing $x$. Write $\Gamma\langle -x+1 \rangle$ for the graph obtained from $\Gamma$ by removing the marking $x$ and adding $1$ to the genus of the vertex that contained $x$. 

Let
\[w = \sum_{\Gamma : g(x) \geq g - 1 - g(x)} a_{\Gamma} [\Gamma\langle -x+1\rangle ] + \sum_{i=1}^n c_i \psi_i + d \kappa_1.\]
Recall that each of the $\psi$ classes pulls back to the $\psi$ class of the same name.
Note also that $\varphi^*\delta_{1,\emptyset} = \delta_{1,\emptyset} - \psi_x$.
By construction
\begin{equation} \label{bc} u - \varphi^*w = c_x' \psi_x + \sum_{\Gamma : g(x) < g - 1 - g(x)} b_{\Gamma} [\Gamma] 
\end{equation}
for some $b_{\Gamma}$ and $c_x' =c_x + a_{\Gamma_e}$, where $\Gamma_e$ is the elliptic tail graph.
%\johannes{Shouldn't it be $c_x' =c_x + a_{\Gamma_e}$ because of two minusses cancelling?}

We have $\alpha^*[\Gamma]$ is the sum of graphs where we decrease the genus of one of the vertices by one and add the marking $y$ on that vertex. Write $\Gamma\langle y \neq x \rangle$ for the term in $\alpha^*[\Gamma]$ where $y$ and $x$ are not on the same vertex. Then we have
\[\alpha^*(u - \varphi^*w) = c_x' \psi_x + \sum_{\substack{\Gamma \\ 2g(x) < g-1}} b_{\Gamma} [\Gamma\langle y \neq x \rangle] + \text{terms with $x, y$ on same vertex}.\]
Let $\tau$ be the automorphism of $\M_{g-2,n+2}^{\ct}$ that swaps $x$ and $y$.
Then,
\begin{equation} \label{diff} 0 = (\alpha^* - \beta^*)(u - \varphi^*w) = c_x'(\psi_x - \psi_y) + \sum_{\substack{\Gamma \\ 2g(x) < g - 1}} b_{\Gamma}[\Gamma\langle y \neq x \rangle] - \sum_{\substack{\Gamma \\ 2g(x) < g - 1}} b_{\Gamma} \tau^*[\Gamma\langle y \neq x \rangle]. 
\end{equation}
The left-hand side vanishes because $u$ and $\varphi^*w$ are both in $\ker(\alpha^* - \beta^*)$.

If $g - 1$ is even, then $2g(x) < g - 1$ means that $g(x) < (g - 1)/2$, so $g - 1 - g(x) \geq g(x) + 2$. This means that $g(y) > g(x)$ in the graph $\Gamma \langle y \neq x \rangle$. 
Since $g(y) > g(x)$, the graphs in the second sum in \eqref{diff} are distinct from those in the first. Assuming that $g - 2 \geq 2$, all terms on the right hand side of \eqref{diff} are independent by \cite[Theorem 2.2]{ArbarelloCornalba}.
Thus, we have $c_x' = 0$ and $b_{\Gamma} = 0$ for all $\Gamma$. Hence, considering \eqref{bc}, we have $u - \varphi^*w = 0$ so $u$ lies in the image of $\varphi^*$.

Now suppose $g - 1$ is odd and that $g - 2 \geq 2$.
As in the previous paragraph, we learn that $c_x' = 0$ and $b_{\Gamma} = 0$ for $\Gamma$ with $g(x) < (g - 2)/2$. However, when $g(x) = (g - 2)/2$, then there is another graph $\Gamma'$ with the property that $\tau(\Gamma'\langle y \neq x \rangle) = \Gamma$ and vanishing of \eqref{diff} implies $b_{\Gamma} = b_{\Gamma'}$, rather than that both vanish. 
If $\Gamma$ has markings $x \cup A$ on the vertex of genus $g/2 - 1$, then $\Gamma'$ is the graph with $x \cup A^c$ on the vertex of genus $g/2 - 1$.
In particular, $[\Gamma] + [\Gamma'] = \varphi^*[\Gamma \langle - x + 1 \rangle]$. Hence, considering \eqref{bc} we see that $u - \varphi^* w$ is actually in the image of $\varphi^*$, so $u$ lies in the image of $\varphi^*$.

Finally, we treat the case $g -1 = 2$.
Recall that $\varphi^*\kappa_1 = \kappa_1$, but in genus $2$ there is a relation that expresses $\kappa_1$ in terms of $\psi$ and boundary classes \cite[Theorem 2.2(b)]{ArbarelloCornalba}. Let us define $\epsilon_\Gamma$ to be $-1$ if $\Gamma$ has a vertex of genus $0$ and $\frac{7}{5}$ if $\Gamma$ has a vertex of genus $1$. Then the relation is
\[\kappa_1 = \psi_x + \sum_{i=1}^{n} \psi_i + \sum_{\Gamma} \epsilon_\Gamma [\Gamma]  \in \R^1(\M_{2,n+1}^{\ct}).\]
When $g -1 = 2$, we can write $u$ uniquely as a sum of compact type graphs and $\psi$ classes
\[u = \sum_{\Gamma} a_{\Gamma} [\Gamma] + \sum_{i = 1}^{n} c_i \psi_i + c_x \psi_x  \in \R^1(\M_{2,n+1}^{\ct}).\]
Now, let us set
\[w = c_x\kappa_1 + \sum_{i=1}^n (c_i - c_x) \psi_i + \sum_{\Gamma : g(x) \geq g - 1 - g(x)} (a_{\Gamma} - c_x \epsilon_\Gamma) [\Gamma\langle -x+1\rangle ] \in \R^1(\M_{3,n}^{\ct}).\]
Then,
\begin{equation} \label{bc2} u - \varphi^*w = \sum_{\Gamma : g(x) < g - 1 - g(x)} b_{\Gamma} [\Gamma]
\end{equation}
for some $b_\Gamma$.
Notice that $g(x) < 2 - g(x)$ implies $g(x) = 0$, so the sum is over graphs with $g(x) = 0$.
 The rest of the proof proceeds similarly. We find that
\begin{equation} \label{forg2} 0 = (\alpha^* - \beta^*)(u - \varphi^*w) = \sum_{\Gamma} b_{\Gamma} [\Gamma \langle y \neq x \rangle] - \sum_{\Gamma} b_{\Gamma} [\Gamma \langle y \neq x \rangle].
\end{equation}
The first sum consists of graphs where $g(x) = 0$ and $g(y) = 2$ while the second sum consists of graphs where $g(x) = 2$ and $g(y) = 0$. In genus $1$, there are no relations among the boundary divisors. Thus, the terms on the right are independent, so all $b_{\Gamma} = 0$. Considering \eqref{bc2}, we see that $u - \varphi^*w = 0$, so $u$ lies in the image of $\varphi^*$.
\end{proof}

%Meanwhile, using that the pairing is perfect for smaller $g$ and \eqref{pushpull}, we find that 
%\[\alpha_* - \beta_*: R^{2g - 6 +n}(\M_{g-2,n+2}^{\ct}) \to R^{2g - 5 + n}(\M_{g-1,n+1}^{\ct})\]
%is the dual of $\alpha^* - \beta^*$. Thus,
%Thus,
%\[(\im \varphi^*)^\perp = \ker(\alpha^* - \beta^*)^\perp = \im(\alpha_* - \beta_*) \subset \ker \varphi_*,\]
%where the last containment follows form commutativity of \eqref{beq}.

\begin{proof}[Proof of Proposition \ref{Rsminus1part} for $g\geq 3$ assuming the $g=2$ case]
Suppose $v\in \R_{\FZ}^{2g-4+n}(\M_{g,n}^{\ct})$. By Lemma \ref{pushforwardsurjective}, we can write $v=\varphi_*\tilde{v}$ for some $\tilde{v} \in \R_{\FZ}^{2g - 5+n}(\M_{g-1,n+1}^{\ct})$. For any $w\in \R_{\FZ}^{1}(\M_{g,n}^{\ct})$, we have
\[
v\cdot w = \varphi_*\tilde{v}\cdot w=\tilde{v}\cdot \varphi^*w
\]
by \eqref{pushpull}. We will show that $(\im \varphi^*)^{\perp}\subset \ker \varphi_*$, so that 
$v\cdot w=0$ for all $w$ only if $v=0$.

    By induction on $g$, the pairing induces isomorphisms
    \[
    \R^1(\M_{g-2,n+2}^{\ct})^{\vee}\cong \R_{\FZ}^{2g-6+n}(\M_{g-2,n+2}^{\ct}) \quad \text{ and } \quad \R^1(\M_{g-1,n+1}^{\ct})^{\vee}\cong \R_{\FZ}^{2g-5+n}(\M_{g-1,n+1}^{\ct}).
    \]
   By \eqref{pushpull}, it follows that under these dualities, the map 
   \[
   \alpha_*-\beta_*:\R_{\FZ}^{2g-6+n}(\M_{g-2,n+2}^{\ct})\rightarrow \R_{\FZ}^{2g-5+n}(\M_{g-1,n+1}^{\ct})
   \]
   is dual to the pullback map 
   \[
   \alpha^*-\beta^*:\R_{\FZ}^{1}(\M_{g-1,n+1}^{\ct})\rightarrow \R_{\FZ}^{1}(\M_{g-2,n+2}^{\ct}).
   \]
By Lemma \ref{cx} and duality,
\[(\im \varphi^*)^\perp = \ker(\alpha^* - \beta^*)^\perp = \im(\alpha_* - \beta_*) \subset \ker \varphi_*,\]
where the last containment follows from the commutativity of \eqref{beq}.
\end{proof}
Theorem \ref{socleminus1} will thus follow once we prove Propositions \ref{R1part} and \ref{Rsminus1part} when $g=2$.
\subsubsection{Genus 2}
The genus $2$ analogue of Lemma \ref{pullbackinjective} is the following.
\begin{lem} \label{g2ker}
If $g = 2$, then the kernel of $\varphi^*: \R^1(\M_{2,n}^{\ct}) \to \R^1(\M_{1,n+1}^{\ct})$ is the span of $\lambda_1$. 
\end{lem}
\begin{proof}
We have $0\neq \lambda_1\in \R^1(\M_2^{\ct})$ and $\R^1(\M_2^{\ct}) \to \R^1(\M_{1,1}) = 0$, so $\lambda_1$ is in the kernel of $\varphi$ with no markings. Now assume $n > 0$. Considering the commutative diagram
\begin{center}
\begin{tikzcd}
\R^1(\M_{2,n}^{\ct}) \arrow{r}{\varphi^*} & \R^1(\M_{1,n+1}^{\ct}) \\
\R^1(\M_2^{\ct}) \arrow{u} \arrow{r} & \R^1(\M_{1,1}) \arrow{u}
\end{tikzcd}
\end{center}
we see that $\lambda_1$ must lie in the kernel of $\varphi^*$ for all $n$.

It suffices to show that $\varphi^*|_W$ is injective for some codimension $1$ subspace $W \subset \R^1(\M_{2,n}^{\ct})$. We take
$W$ to be the subspace spanned by $\psi$ classes and boundary divisors besides $\delta_{1, \emptyset}$. 
Let
\begin{equation} \label{vo} w =  \sum_{\substack{1 \in S  \\ S^c \neq \emptyset}} a_S \delta_{1,S}
+ \sum_{|S| \geq 2} b_S \delta_{0,S} + \sum_{i=1}^n c_i \psi_i
 \in W,
 \end{equation}
where the sums run over subsets $S \subset \{1, \ldots, n\}$. The pullback is
\begin{align} 
\varphi^*w &= \sum_{\substack{1 \in S \\ S^c \neq \emptyset}}a_S (\delta_{1,S} + \delta_{0, S + x}) + \sum_{|S|\geq 2}b_S \delta_{0, S} + \sum_{i=1}^n c_i \psi_i, \notag \\
\intertext{where the sum still runs over sets $S \subset \{1, \ldots, n\}$. Using the relations $\delta_{1,A} = \delta_{0, A^c}$ and $\psi_i = \sum_{i \in A, |A| \geq 2} \delta_{0,A}$ in $\R^1(\M_{1,n+1}^{\ct})$, we rewrite this in terms of the boundary divisors
$\delta_{0,A}$:}
\varphi^* w&=\sum_{\substack{1 \in S \\ S^c \neq \emptyset}} a_{S} (\delta_{0,S^c + x}
+
\delta_{0, S + x}) + \sum_{|S|\geq 2}b_S \delta_{0, S} \label{phiv}\\
&\qquad + \sum_{i=1}^n c_i \left(\sum_{i \in A, |A| \geq 2} \delta_{0, A} \right), \notag
\end{align}
where the sums run over sets $S \subset \{1, \ldots, n\}$ and the last sum over sets $A \subset \{1, \ldots, n, x\}$.

Now suppose that $\varphi^*w = 0$.
Since the boundary divisors $\delta_{0,A}$ form an independent set in $\R^1(\M_{1,n+1}^{\ct})$, when we collect terms, the coefficient of each $\delta_{0,A}$ above vanishes. We use this to prove that all $a_S, b_S$ and $c_i$ vanish.
Take some set $S \subset \{1,\ldots, n\}$ with $1 \in S$ and $S^c \neq \emptyset$. Considering the coefficient of $\delta_{0,S +x}$ in \eqref{phiv}, we have
\begin{equation} \label{as} 0 = a_S + \sum_{i \in S} c_i,
\end{equation}
while from considering the coefficient of $\delta_{0,S^c+x}$, we have
\[0 = a_S + \sum_{i \in S^c} c_i.\]
Hence, for every set with $1 \in S$ and $S^c \neq 0$, we have
\[\sum_{i \in S} c_i = \sum_{i \in S^c} c_i.\]
In addition, considering the coefficient of $\delta_{0, \{1, \ldots, n,x\}}$, we see
\[0 = \sum_{i=1}^n c_i.\]
This implies all $c_i = 0$. (Indeed, $c_i =\sum_{j \neq i}c_k$ and
$0 = \sum c_k$ forces $c_i = -c_i$.)
Hence, \eqref{as} implies all $a_S = 0$. Finally, given $S \subset \{1, \ldots, n\}$ with $|S| \geq 2$, considering the coefficient of $\delta_{0,S}$ shows
\[0 = b_S + \sum_{i \in S} c_i,\]
which implies $b_S = 0$ as well.
Thus, considering \eqref{vo}, we have $w = 0$.
\end{proof}
\begin{proof}[Proof of Proposition \ref{R1part} when $g=2$] 
    Let $w\in \R^1(\M_{2,n}^{\ct})$ be nonzero and write
    \[
    w = c\lambda_1 + w'
    \]
    where $w'$ does not lie in the span of $\lambda_1$. First we assume $w'\neq 0$. Then by Lemma \ref{g2ker}
    \[
    \varphi^* w = \varphi^* w'\neq 0.
    \]
    By the Gorenstein property in genus $1$ \cite{Tavakol}, there exists $\tilde{v}\in\R_{\FZ}^{n-1}(\M_{1,n+1}^{\ct})$ such that \[
    0 \neq \tilde{v}\cdot \varphi^*w=\varphi_*\tilde{v}\cdot w,
    \]
    where the equality is \eqref{pushpull}. 
    
    Now we assume $w'=0$ and $c\neq 0$. Let 
    \[
    \xi: \M_{2,1}^{\ct}\times \M_{0,n+1}^{\ct}\rightarrow \M_{2,n}^{\ct}
    \]
    be the map gluing the last marked point on each component. Set $z=\xi_*(\psi\otimes \mathrm{pt})\in \R_{\FZ}^{n}(\M_{2,n}^{\ct})$. Then
    \begin{equation} \label{zpair} w \cdot z =
    c\lambda_1\cdot z = c\xi_*(\xi^*\lambda_1\cdot (\psi\otimes \mathrm{pt}))=c\xi_*(\lambda_1\psi_1\otimes \mathrm{pt}).
    \end{equation}
    Above, $\lambda_1 \psi \in \R_{\FZ}^2(\M_{2,1}^{\ct}) \cong \qq$ lies in the socle degree and is non-zero, which can be seen by lifting $\lambda_1\psi$ to $\Mb_{2,1}$ and pairing with $\lambda_2$. By the geometric description of the socle in Section \ref{yesgorensteinsection}, it is clear that the map $\xi_*$ sends the generator of the socle of $\M_{2,1}^{\ct}$ to the generator of the socle of $\M_{2,n}^{\ct}$, so \eqref{zpair} is nonzero.
\end{proof}
Next we prepare to prove Proposition \ref{Rsminus1part} in genus $2$. The $g=2$ analogue of Lemma \ref{pushforwardsurjective} is the following.
\begin{lem}\label{g2pushforwardsurjective}
If $n\geq 1$, the group $\R_{\FZ}^{n}(\M_{2,n}^{\ct})$ is spanned by the image of 
\[\R_{\FZ}^{n-1}(\M_{1,n+1}^{\ct}) \cong \R_{\FZ}^{n-1}(\M_{1,n+1}^{\ct})\otimes \R_{\FZ}^0(\M^{\ct}_{1,1}) \to \R_{\FZ}^{n}(\M_{2,n}^{\ct})\]
together with the pushforward of $\psi \otimes \mathrm{pt} \in \R_{\FZ}^{n-1}(\M_{2,1}^{\ct}) \otimes \R_{\FZ}^0(\M^{\ct}_{0,n+1}) \to \R_{\FZ}^{n}(\M_{2,n}^{\ct})$. 
\end{lem}
\begin{proof}
By Lemma \ref{theoremstarconsequence}, classes in codimension $2g - 4 + n = n$ are generated by decorated graphs where all but one vertex have $(g(v), n(v)) = (0,3)$ or $(1,1)$ and 
one vertex has $(g(v), n(v)) = (0,4), (1,2)$ or $(2,1)$. If there are no $(2, 1)$ vertices, then such a graph has a $(1,1)$ vertex.

Meanwhile, there is one such graph that has no elliptic tails, namely
we take a $(2,1)$ vertex decorated with a codimension $1$ decoration and the rest of the graph is $(0,3)$ vertices. We can take the decoration to be $\psi_1$ because $\kappa_1$ is proportional to $\psi_1$ modulo boundary divisors on $\M_{2,1}^{\ct}$ by \cite[Theorem 2.2(b)]{ArbarelloCornalba}.
\end{proof}

 Lemma \ref{cx} holds when $g=2$ as well.
\begin{lem} \label{cx2}
When $g=2$ and $n \geq 2$, the complex \eqref{2tails} is exact.
\end{lem}
%\hannah{It looks like the argument for peeling off the coefficient of $\delta_{0,\{1, \ldots, n,x\}}$ only works when $n \geq 2$, so I added that to the statement and we can say the cases $g = 2, n \leq 2$ have been verified by computer in the proof of Prop 17}
\begin{proof}
The idea is similar to the proof of Lemma \ref{cx}, except there are more relations in low genus. Suppose for contradiction that there exists $u \in \ker(\alpha^* - \beta^*) \subset  \R^1(\M_{1,n+1}^{\ct})$ such that $u \notin \im \varphi^*$.
Since $\R^1(\M_{1,n+1}^{\ct})$ is spanned by the boundary divisors $\delta_{0, S}$, we can write
\[u = \sum_{|S| \geq 2} c_S \delta_{0,S}.\]
Note that $\varphi^*\psi_i = \psi_i = \sum_{i \in S} \delta_{0,S} = \delta_{0, \{i,x\}} + \ldots$.
Replacing $u$ with $u -  \sum_{i = 1}^n c_{\{i,x\}}\varphi^*\psi_i$, we can assume $c_{S} = 0$ when $x \in S$ and $|S| = 2$. Next, observe that, when written in terms of boundary divisors, \[\varphi^*\left(\delta_{1,\emptyset} + \sum_{i=1}^n \psi_i\right)= \delta_{1, \emptyset} - \psi_x + \sum_{i=1}^n \psi_i = (n-1)\delta_{0, \{1, \ldots, n, x\}} + \ldots\]
has coefficient $n-1$ on $\delta_{0, \{1, \ldots, n, x\}}$ and coefficient $0$ on all $\delta_{0, \{i,x\}}$. Thus, by replacing $u$ with $u - \frac{1}{n-1}c_{\{1, \ldots, n, x\}} \varphi^*\left(\delta_{1, \emptyset} + \sum_{i=1}^n \psi_i \right)$, we can assume $c_S = 0$ when $x \in S$ and $|S| = 2$ and when $S = \{1, \ldots, n,x\}$.
Next, replacing $u$ with
\[u - \sum_{x,1 \in  S}c_S \cdot \varphi^*\delta_{1, S - x}\]
we can assume that $c_S = 0$ when $x, 1 \in S$, or $|S| = 2$, or $S = \{1, \ldots, n,x\}$. (Note that if $1 \in S$, and $S \neq \{1, \ldots, n, x\}$, then $\varphi^*\delta_{1,S - x} = \delta_{0, S} + \delta_{1,S} = \delta_{0, S} + \delta_{0, S^c + x}$; since $1 \notin S^c + x$, this second term does not change other coefficients we have already fixed to be zero.)
Finally, replacing $u$ with
\[u - \sum_{x \notin S} c_S \cdot \varphi^* \delta_{0,S}\]
we may also assume that $c_S = 0$ when $x \notin S$.

In summary, after adjusting $u$ by elements in $\im \varphi^*$, 
we may assume that $u$ has the form
\[u = \sum_{\substack{x \in S, 1 \notin S \\ |S| \geq 3 \\ S^c \neq \emptyset}} c_S \delta_{0,S}.\]
Now, since $u \in \ker(\alpha^* - \beta^*)$, we have
\[0 = (\alpha^* - \beta^*)(v) = \sum_{\substack{x \in S, 1 \notin S \\ |S| \geq 3 \\ S^c \neq \emptyset}} c_S \delta_{0, S} - \sum_{\substack{x \in S, 1 \notin S \\ |S| \geq 3 \\ S^c \neq \emptyset}} c_S \delta_{0, S -x + y}.\]
Note that in the terms $\delta_{0,P}$ appearing above, $|P| \geq 3$ and $1 \in P^c$. 
On $\Mb_{0,n+2}$, the $\delta_{0,P} = \delta_{0, P^c}$ with $1 \in P^c$ and $|P| \geq 3$ are independent by \cite[Lemma 3.9]{ArbarelloCornalba}. It follows that all $c_S = 0$. Thus, $u = 0 \in \im \varphi^*$, which is a contradiction.
\end{proof}

\begin{proof}[Proof of Proposition \ref{Rsminus1part} when $g=2$]
We have computationally checked the cases $n \leq 1$, so assume $n \geq 2$.
Let $v\in \R_{\FZ}^n(\M_{2,n}^{\ct})$. By Lemma \ref{g2pushforwardsurjective}, we can write 
\[v=\varphi_*\tilde{v}+c\xi_*(\psi\otimes \mathrm{pt}),\] where $\tilde{v}\in \R_{\FZ}^{n-1}(\M_{1,n+1}^{\ct})$ and $c$ is a constant. Note that $\varphi_*\tilde{v}\cdot\lambda_1= \tilde{v} \cdot\varphi^*\lambda_1=0$. Therefore, 
\[
v\cdot \lambda_1=c\xi_*(\psi\otimes \mathrm{pt})\cdot \lambda_1.
\]
If $c\neq 0$, then $v\cdot \lambda_1\neq 0$ by the same argument as in the $g=2$ case of the proof of Proposition \ref{R1part}. Now we can assume $c=0$, and in this case the proof is exactly the same as in the case when $g\geq 3$, using Lemma \ref{cx2}.
\end{proof}

\noindent Combining Propositions \ref{R1part} and \ref{Rsminus1part}
shows that the pairing
\[
\R^1_{\FZ}(\M_{g,n}^{\ct})\times \R_{\FZ}^{2g-4+n}(\M_{g,n}^{\ct})\rightarrow \R^{2g-3+n}_{\FZ}(\M_{g,n}^{\ct})\cong \qq
\]
is perfect. The perfectness of the pairing forbids further relations, thereby proving Theorems \ref{Pixtonminus1} and \ref{socleminus1}.
\qed

\section{The Gorenstein property and invisibility} \label{prelim}
%Recall that we say that $\R^*(\M_{g,n}^{\ct})$ is Gorenstein if the pairing
%\begin{equation}\label{pairing}
%    \R^i(\M_{g,n}^{\ct})\times \R^{2g-3+n-i}(\M_{g,n}^{\ct})\rightarrow \R^{2g-3+n}(\M_{g,n}^{\ct})\cong \qq
%\end{equation}
%is perfect for all $i$, and likewise for the analogous statement in cohomology. 
In this section, we study some generalities about the failure of the Gorenstein property. The failure of the Gorenstein property in Chow follows from its failure in cohomology, by Proposition \ref{chowvscoh}. We will thus work in cohomology here.
\begin{definition}
    We say that a nonzero class $\alpha \in \R\H^{2i}(\M_{g,n}^{\ct})$ is \emph{invisible} if
    \[
    \alpha\cdot \beta = 0
    \]
    for all $\beta \in \R\H^{4g-6+2n-2i}(\M_{g,n}^{\ct})$.
\end{definition}
\noindent The tautological ring is not Gorenstein if and only if there exists an invisible class.

 We show that certain natural operations on compact type moduli spaces send invisible classes to invisible classes. This structure arises because of certain operations that preserve the socle.
Let $\pi:\M_{g,n+1}^{\ct}\rightarrow \M_{g,n}^{\ct}$ be the map forgetting the last marked point.
Recall that the socle $\RH^{4g-6+2n}(\M_{g,n}^{\ct})$ is generated by the fundamental class of the locus parametrizing maximally degenerate $n+g$ pointed rational curves with $g$ elliptic tails. It follows that
\begin{equation}\label{forgetpointiso}
\pi_*: \R\H^{4g-4+2n}(\M_{g,n+1}^{\ct})\xrightarrow{\cong} \R\H^{4g-6+2n}(\M_{g,n}^{\ct})
\end{equation}
is an isomorphism.

\begin{lem}\label{morepoints}
    Suppose $\alpha\in \R\H^{2i}(\M_{g,n}^{\ct})$ is invisible. Then $\pi^*\alpha \in \R\H^{2i}(\M_{g,n+1}^{\ct})$ is invisible.
\end{lem}
\begin{proof}
Because the Gorenstein property holds for $g=0,1$, we can assume $g\geq2$.  For any $\gamma\in \R\H^{2i}(\M_{g,n}^{\ct}),$ we have
\[
 \gamma = \frac{1}{2g -2+n} \pi_*((\pi^* \gamma) \cdot \psi_{n+1}).
\]
Therefore, $\pi^*$ is injective. Let $\beta\in \R\H^{4g-4+2n-2i}(\M_{g,n+1}^{\ct})$. We have
\[
\pi_*(\pi^*\alpha\cdot \beta)=\alpha\cdot \pi_*\beta=0,
\]
where the second equality follows from the assumption that $\alpha$ is invisible.
 Because $\pi^*\alpha$ is not zero and $\pi^*\alpha\cdot \beta$ is in the socle degree, the claim follows from the fact that the map in \eqref{forgetpointiso} is an isomorphism.
 %By \cite[Section 4.1.2]{FaberPandharipande}, the socle for $\M_{g,n}^{\ct}$ is represented by the locus of maximally degenerate $g+n$ pointed rational curves with $g$ elliptic tails attached to $g$ marked points. Looking at these representatives, we see the pushforward is nonzero.
\end{proof}

Invisible classes also play well with pushforward along gluing maps
\[\varphi: \M_{g,n}^{\ct} \times \M_{g',n'}^{\ct} \to \M_{g+g',n+n'-2}^{\ct}.\]
Given $\alpha \in \H^*(\M_{g,n}^{\ct})$ and $\gamma \in \H^*(\M_{g',n'}^{\ct})$, we write $\alpha \otimes \gamma \in \H^*(\M_{g,n}^{\ct} \times \M_{g',n'}^{\ct})$ for the product of the pullbacks of these two classes along the projection maps. For $\beta\in \H^*(\M_{g,n}^{\ct}\times \M^{\ct}_{g',n'})$, we write $\beta \in \RH^*(\M_{g,n}^{\ct}\times \M^{\ct}_{g',n'})$ if it admits a tautological K\"unneth decomposition.
Because the tautological ring vanishes beyond the socle degree, we have
\[   \R\H^{4(g+g') - 12 + 2(n+n')}(\M^{\ct}_{g,n} \times \M^{\ct}_{g',n'}) \cong \R\H^{4g - 6 + 2n}(\M^{\ct}_{g,n}) \otimes 
    \R\H^{4g' - 6 + 2n'}(\M^{\ct}_{g',n'}) \cong \qq.
\]
When we glue together two maximally degenerate genus $0$ curves, the result is a maximally degenerate genus $0$ curve.
Thus, using the geometric description of the socle, we see that the pushforward map
\begin{equation}\label{gluingmap}
  \varphi_*: 
\R\H^{4g - 6 + 2n}(\M^{\ct}_{g,n}) \otimes 
    \R\H^{4g' - 6 + 2n'}(\M^{\ct}_{g',n'}) 
 \to
\R\H^{4(g'+n')-6+2(n+n'-2)}(\M_{g+g',n+n'-2}^{\ct}),
\end{equation}
is an isomorphism.
If $\alpha \in \R\H^{2i}(\M_{g,n}^{\ct})$ is invisible, it readily follows that, \emph{if it is nonzero}, then $\varphi_*(\alpha \otimes \gamma)$ is invisible for any $\gamma \in \R\H^*(\M_{g',n'}^{\ct})$. Indeed, 
for any $\beta$ in complementary degree to $\varphi_*(\alpha \otimes \gamma)$, we have
\[\varphi_*(\alpha \otimes \gamma) \cdot \beta = \alpha \otimes \gamma \cdot \varphi^*\beta  
= 0.\]
Above, the first equality follows from the fact that \eqref{gluingmap} is an isomorphism;
the second equality follows because the only nonzero terms come from K\"unneth components of $\varphi^*\beta$ where the first factor lies in complementary degree to $\alpha$, and we are assuming $\alpha$ is invisible.

In general, it may be difficult to determine when the pushforward 
$\varphi_*(\alpha \otimes \gamma)$ is non-zero. One sufficient criterion is if 
$\varphi^*\varphi_*(\alpha \otimes \gamma) \neq 0$. This class can be computed by considering the fiber product of the gluing map with itself and using the excess intersection formula:
\begin{equation} \label{pp} \varphi^*\varphi_*(\alpha \otimes \gamma) = \begin{cases} 
-(\alpha \psi_p) \otimes \gamma - \alpha \otimes (\psi_{p'} \gamma) 
& \text{if $n' > 1$} \\
(\alpha (\delta_{g', \emptyset} -\psi_p)) \otimes \gamma - \alpha \otimes (\psi_{p'} \gamma)
&\text{if $n' = 1$.}
\end{cases}
\end{equation}
One case where we can see such classes are non-zero is when $\gamma=1$ and $(g',n') \neq (1,1)$, so that $\psi_{p'} \neq 0$. %Taking $\gamma = 1$ above, we see that if $\alpha$ is invisible, then $\varphi_*(\alpha\otimes 1)$ is also invisible.
Another case where we can verify the above pushforward is non-zero is when $\alpha$ is pulled back from a moduli space with less markings. This is the idea behind the following result.

\begin{lem} \label{gluelem}
Let $\varphi$ be the gluing map  \eqref{gluingmap} that glues $p$ and $p'$
and let $\pi: \M_{g,n}^{\ct} \to \M_{g,n-1}^{\ct}$ be the map that forgets the marking $p$.
Suppose $\alpha \in \R\H^{2i}(\M_{g,n-1}^{\ct})$ is invisible. Then for any nonzero $\gamma \in \R\H^{2j}(\M_{g',n'}^{\ct})$, we have $\varphi_*(\pi^*\alpha \otimes \gamma) \in \R\H^{2i + 2j + 2}(\M_{g+g',n+n'-2})$ is invisible.
\end{lem}
\begin{proof}
   We can assume $g\geq 2$, as otherwise the assumption is vacuous. By the discussion above, it suffices to show that
   $\varphi^*\varphi_*(\pi^*\alpha \otimes \gamma) \neq 0$. By \eqref{pp} (with $\alpha$ replaced by $\pi^*\alpha$) it suffices to show that $\pi^*\alpha \cdot \psi_p \neq 0$ if $n' > 1$ or $\pi^*\alpha \cdot (\delta_{g',\emptyset} - \psi_p) \neq 0$ if $n' = 1$.
To see this class is non-zero, we consider its pushforward along $\pi$: in either case, we obtain a non-zero multiple of $\alpha$. 
\end{proof}

Applying Lemma \ref{gluelem} when $(g', n') = (1,1)$ and $\gamma = 1$, we obtain:
\begin{lem}\label{moregenus}
If $\R\H^*(\M_{g-1,n}^{\ct})$ is not Gorenstein, then $\R\H^*(\M_{g,n}^{\ct})$ is not Gorenstein.
\end{lem}

In order to prove that $\RH^*(\M_{g,n}^{\ct})$ is not Gorenstein for $2g+n\geq 12$, it suffices to show it is not Gorenstein in the cases $2g+n-12=0$ by Lemmas \ref{morepoints} and \ref{moregenus}. This idea is illustrated in Figure \ref{proofscheme}.
Indeed, once invisible classes are found at the pairs $(g,n)$ marked by red boxes, we apply Lemma
\ref{moregenus} (indicated by the purple horizontal arrow below) followed by
Lemma \ref{morepoints} (indicated by red vertical arrows) to see that all boxes above and to the right are also red. 
The case $(g,n)=(2,8)$ (dark red) was established by Petersen \cite{Petersen}. The remaining cases will be dealt with in Section \ref{2gn12}.
\begin{figure}
\begin{tikzpicture}
\begin{axis}[
    xmin=-0, xmax=9.5,
    ymin=-0, ymax=12,
    axis lines=center,
    axis line style={->},
    xlabel={$g$},
    ylabel={$n$},
    xtick={0,1,...,9},
    ytick={0,1,...,11},
    xticklabels={-1,...,8},
    yticklabels={-1,...,11},
    xticklabel style={xshift=-0.35cm},
    yticklabel style={yshift=-0.25cm},
    grid=both,
    grid style={line width=.1pt, draw=black},    
    minor tick num=0,
    enlargelimits={abs=0.5},
    axis on top,
    clip=false
]
% Black tiles
\foreach \i in {0,...,2} {
    \addplot [fill=black!10!white] coordinates {(0,\i) (0,\i+1) (1,\i+1) (1,\i)};
}
\foreach \i in {0} {
    \addplot [fill=black!10!white] coordinates {(1,\i) (1,\i+1) (2,\i+1) (2,\i)};
}

% Green Tiles
\foreach \i in {3,...,10} {
    \addplot [fill=green!30!white] coordinates {(0,\i) (0,\i+1) (1,\i+1) (1,\i)};
}
\foreach \i in {1,...,10} {
    \addplot [draw=black,fill=green!30!white] coordinates {(1,\i) (1,\i+1) (2,\i+1) (2,\i)};
}
\foreach \i in {0,...,7} {
    \addplot [draw=black,fill=green!30!white] coordinates {(2,\i) (2,\i+1) (3,\i+1) (3,\i)};
}
\foreach \i in {0,...,5} {
    \addplot [draw=black,fill=green!30!white] coordinates {(3,\i) (3,\i+1) (4,\i+1) (4,\i)};
}
\foreach \i in {0,...,3} {
    \addplot [draw=black,fill=green!30!white] coordinates {(4,\i) (4,\i+1) (5,\i+1) (5,\i)};
}
\foreach \i in {0,...,1} {
    \addplot [draw=black,fill=green!30!white] coordinates {(5,\i) (5,\i+1) (6,\i+1) (6,\i)};
}

% dark red boxes
\foreach \i in {0,...,3} {
    \addplot [draw=black,fill=red!80!black] coordinates {(3+\i,6-2*\i) (3+\i,6-2*\i+1) (4+\i,6-2*\i+1) (4+\i,6-2*\i)};
}

\foreach \i in {-1} {
    \addplot [draw=black,fill=red!50!black] coordinates {(3+\i,6-2*\i) (3+\i,6-2*\i+1) (4+\i,6-2*\i+1) (4+\i,6-2*\i)};
}

\draw[color=red, thick, ->] (2.5, 8.5) -- (2.5, 11.5);
\draw[color=red, thick, ->] (3.5, 6.5) -- (3.5, 11.5);
\draw[color=red, thick, ->] (4.5, 4.5) -- (4.5, 11.5);
\draw[color=red, thick, ->] (5.5, 2.5) -- (5.5, 11.5);
\draw[color=red, thick, ->] (6.5, 0.5) -- (6.5, 11.5);
\draw[color=red, thick, ->] (7.5, 0.5) -- (7.5, 11.5);
\draw[color=red, thick, ->] (8.5, 0.5) -- (8.5, 11.5);
\draw[color=violet, thick,->] (6.5, 0.5) -- (9.5, .5);

\end{axis}
\end{tikzpicture}
\caption{Lemmas \ref{morepoints} and \ref{moregenus} reduce the proof of Theorem \ref{Gortheorem} to the cases when $2g+n=12$.}\label{proofscheme}
\end{figure}
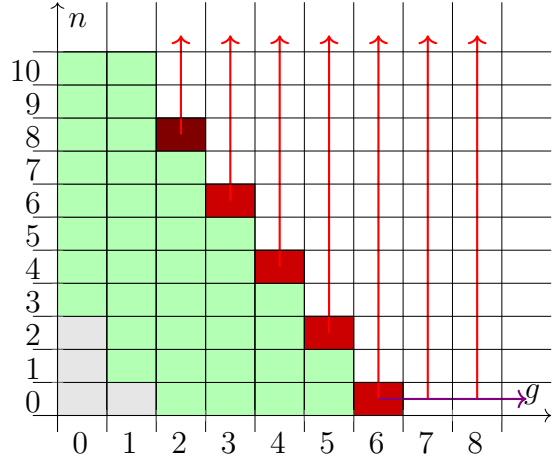

\section{The tautological ring when \texorpdfstring{$2g+n=12$}{2g+n=12}}\label{2gn12}
Here, we show that $\R\H^*(\M_{g,n}^{\ct})$ is not Gorenstein for $(g,n)=(6,0), (5,2), (4,4),$ and $(3,6)$, thereby proving Theorem \ref{Gortheorem}. Then, we will prove Theorem \ref{Pixtontheorem}. 
%The basic strategy for proving Theorem \ref{Pixtontheorem} is to show that the $3$-spin relations are complete on $\Mb_{g,n}$ for $(g,n)$ in the statement of the Theorem, and then use Lemma \ref{bartoct}. The approach is only possible because the Chow and cohomology ring of $\Mb_{g,n}$ is entirely tautological for small $g$ and $n$ \cite{CL-CKgP}, which implies that the tautological ring of $\Mb_{g,n}$ is Gorenstein.

 %By Proposition \ref{yesgorenstein} and Lemmas \ref{morepoints} and \ref{moregenus}, this will finish the proof of Theorem \ref{Gortheorem}.

\subsection{Genus 5 and 6}
The cases $g=5,6$ are simplest, so we treat them first.
%In \cite[page 122]{Pixton}, Pixton predicted that $\dim \R\H^8(\M_6^{\ct})=71$, but $\dim \R\H^{10}(\M^{\ct}_6)=72$, using the $3$-spin relations. Similarly, he predicted $\dim \R\H^{8}(\M_{5,2}^{\ct})=636$ and $\dim \R\H^{10}(\M_{5,2}^{\ct})=637$ \samc{Pixton did not predict the latter (he only did the mod $\ss_2$ calculation.}.
\begin{prop}\label{5and6}
    The tautological rings $\RH^*(\M_{6}^{\ct})$ and $\RH^*(\M_{5,2}^{\ct})$ are not Gorenstein.
\end{prop}
\begin{proof}
Using \texttt{admcycles}, we calculate
\[
\dim \R_{\FZ}^{4}(\M_{6}^{\ct})=71, \quad \dim \R_{\FZ}^{5}(\M_{6}^{\ct})=72, \quad \dim \R_{\FZ}^{4}(\M_{5,2}^{\ct})=636, \quad \dim \R_{\FZ}^{5}(\M_{5,2}^{\ct})=637.
\]
It thus suffices to show the $3$-spin relations are complete for $\RH^{10}(\M_{6}^{\ct})$ and $\RH^{10}(\M_{5,2}^{\ct})$, which we do using Lemma \ref{bartoct}.

As in Section \ref{yesgorensteinsection}, we calculate upper bounds \[
\dim \R\H^{10}(\Mb_{6})\leq 988\quad \text{ and }\quad \dim \R\H^{10}(\Mb_{5,2})\leq 7147\]
by computing the rank of the matrices of $3$-spin relations modulo a prime $p$.
Next, we compute lower bounds for the ranks of the pairings
\[
\R\H^{10}(\Mb_{6})\times \R\H^{20}(\Mb_6)\rightarrow \qq
\]
and
\[
\R\H^{10}(\Mb_{5,2})\times \R\H^{18}(\Mb_{5,2})\rightarrow \qq.
\]
The ranks are $988$ and $7147$, respectively, and thus the $3$-spin relations are complete for $\RH^{10}(\Mb_{6})$ and $\RH^{10}(\Mb_{5,2})$. Moreover, by \cite[Theorem 1.4]{CL-CKgP}, $\R\H^{8}(\Mb_{5,2})=\H^{8}(\Mb_{5,2})$ and $\R\H^{8}(\Mb_{4,4})=\H^{8}(\Mb_{4,4})$. Applying Lemma \ref{bartoct}, we see that the $3$-spin relations are complete for $\RH^{10}(\M_{6}^{\ct})$ and $\RH^{10}(\M_{5,2}^{\ct})$. %Thus, by Lemma \ref{bartoct}, to prove Theorem \ref{Gortheorem} when $g=5,6$, it suffices to show that the $3$-spin relations are complete for $\R\H^{10}(\Mb_6)$ and $\R\H^{10}(\Mb_{5,2})$. Moreover, by Proposition \ref{awayfrommiddle}, this will prove Theorem \ref{Pixtontheorem} for $(g,n)=(6,0)$ and $(5,2)$.
\end{proof}
\subsection{Genus 3}\label{3}
Here, we study the case $(g,n)=(3,6)$.
\begin{prop}\label{36notGor}
    The tautological ring $\RH^*(\M_{3,6}^{\ct})$ is not Gorenstein.
\end{prop}
In principle, we could follow the same approach as in the proof of Proposition \ref{5and6}. Unfortunately, because of the computational complexity of verifying that the $3$-spin relations are complete for $\R\H^{10}(\Mb_{3,6})$ using the pairing method, we need a different approach to study $\R\H^{10}(\Mb_{3,6})$. 

The cohomology groups $\H^{k}(\Mb_{g,n})$ are $\ss_n$ representations. For a partition $\lambda$ of $n$, we denote by $s_\lambda$ the corresponding $\ss_n$ representation. The $\ss_6$ representation $\H^*(\Mb_{3,6})$ was calculated by Bergstr\"om and Faber \cite{BergstromFaber} and recorded in \cite{BergstromData}. Additionally, by \cite[Theorem 1.4]{CL-CKgP}, $\H^*(\Mb_{3,6})=\R\H^*(\Mb_{3,6})$. From these results, we have the following proposition.
\begin{prop}\label{Mb36}
As an $\ss_6$ representation, $\RH^{10}(\Mb_{3,6})=\H^{10}(\Mb_{3,6})$ is
\begin{multline}
44 s_{2,1^4} +1086 s_{3,1^3}+767 s_{2^2,1^2}+5851s_{4,1^2}+6034s_{3,2,1}\\+1144s_{2^3}+10327s_{5,1}+10389s_{4,2}+4266s_{3,3}+5713s_6.
\end{multline}
\end{prop}
\begin{cor}\label{36invariants}
We have $\dim \R\H^{10}(\Mb_{3,6})^{\ss_2\times \ss_2 \times \ss_2}=80863$.
\end{cor}
\begin{proof}
    To compute the $\ss_2\times \ss_2\times \ss_2$ invariants, we need to find the number of copies of the trivial representation in the restricted representation
    \[
    \mathrm{Res}^{\ss_6}_{\ss_2\times \ss_2 \times \ss_2} \R\H^{10}(\Mb_{3,6}).
    \]
    By Frobenius reciprocity, this number is the same as the inner product of $\R\H^{10}(\Mb_{3,6})$ with the induced representation
    \[
    \mathrm{Ind}^{\ss_6}_{\ss_2\times \ss_2 \times \ss_2} (\mathbf{1}\boxtimes \mathbf{1}\boxtimes \mathbf{1})=    s_{2^3} + 2s_{3,2,1} + s_{3^2} + s_{4,1^2} + 3s_{4, 2} + 2s_{5, 1} + s_{6}.
    \]
    Thus, by Proposition \ref{Mb36}, we have
    \begin{equation}
        \begin{split}
            \dim \R\H^{10}(\Mb_{3,6})^{\ss_2\times \ss_2 \times \ss_2}&=1144+2(6034)+4266+ 5851 +3(10389) + 2(10327) + 5713 \\
            &=80863. \qedhere
        \end{split}
    \end{equation}
\end{proof}

\begin{proof}[Proof of Proposition \ref{36notGor}]
    Because the intersection pairing 
    \[
    \R\H^8(\M_{3,6}^{\ct})\times \R\H^{10}(\M_{3,6}^{\ct})\rightarrow \qq
    \]
    is $\ss_6$-equivariant, it suffices to show that
    \[
    \dim \R\H^8(\M_{3,6}^{\ct})^{\ss_2\times \ss_2 \times \ss_2}\neq \dim\R\H^{10}(\M_{3,6}^{\ct})^{\ss_2\times \ss_2 \times \ss_2}.
    \]
    Using \texttt{admcycles}, we calculate\footnote{These calculations are done by reducing the matrix of relations modulo $p$, which is the reason we obtain only an inequality.} 
    \[\dim \R_{\FZ}^4(\M_{3,6}^{\ct})^{\ss_2\times\ss_2\times\ss_2}\leq 13159 \qquad\text{and} \qquad \R_{\FZ}^5(\Mb_{3,6})^{\ss_2\times\ss_2\times \ss_2} \leq 80863.\] 
Hence, we have
\[80863 = \dim \R\H^{10}(\Mb_{3,6})^{\ss_2 \times \ss_2 \times \ss_2} \leq \dim \R^5_{\FZ}(\Mb_{3,6})^{\ss_2 \times \ss_2 \times \ss_2} \leq 80863,\]
where the first equality is Corollary  \ref{36invariants}. It follows that the above inequalities are all equalities and that the 
$3$-spin relations are complete for $\R\H^{10}(\Mb_{3,6})^{\ss_2\times \ss_2 \times \ss_2}$. By \cite[Theorem 1.4]{CL-CKgP}, we have $\H^{8}(\Mb_{2,8}) = \R\H^8(\Mb_{2,8})$, so Lemma \ref{bartoct} shows that the $3$-spin relations are complete for $\R\H^{10}(\M_{3,6}^{\ct})^{{\ss_2\times \ss_2 \times \ss_2}}$.
    
Meanwhile, computing the rank of the matrix of relations (this time over $\qq$ to obtain the exact rank), we have
\[
\dim \R^5_{\FZ}(\M_{3,6}^{\ct})^{\ss_2\times\ss_2\times\ss_2} =  13160.
\]
Thus, we conclude
 \begin{align*} \dim \R\H^8(\M_{3,6}^{\ct})^{\ss_2 \times \ss_2 \times \ss_2} &\leq 
 \dim \R_{\FZ}^4(\M_{3,6}^{\ct})^{\ss_2\times\ss_2\times\ss_2} \leq 13159\\
 &< 13160 = \dim \R^5_{\FZ}(\M_{3,6}^{\ct})^{\ss_2\times\ss_2\times\ss_2} = \dim \R\H^{10}(\M_{3,6}^{\ct})^{\ss_2 \times \ss_2 \times \ss_2},
 \end{align*}
 and so $\RH^*(\M_{3,6}^{\ct})$ is not Gorenstein.
 %Using \texttt{admcycles}, we calculate that
    %\[
    %\dim \RH^{8}(\M^{\ct}_{3,6})\leq\dim (\S^{4}(\M^{\ct}_{3,6})/\mathsf{FZ}^{4}(\M^{\ct}_{3,6}))^{\ss_2\times \ss_2 \times \ss_2}<?<\dim \R\H^{10}(\M_{3,6}^{\ct})^{{\ss_2\times \ss_2 \times \ss_2}}=??.
    %\] Thus, $\R\H^*(\M_{3,6}^{\ct})$ is not Gorenstein. \samc{Need to calculate these question mark numbers. }
    %\samc{Calculated $\dim \RH^{8}(\M_{3,6}^{\ct})^{\ss_2}\leq 35292$}
\end{proof}
\begin{rem}
    The only reason for taking $\ss_2\times \ss_2\times\ss_2$ invariants in the above proof is computational. We are unable to compute $\dim \R_{\FZ}^5(\Mb_{3,6})$ using \texttt{admcycles}.
\end{rem}

\subsection{Genus 4}
Here, we study the case $(g,n)=(4,4)$.
\begin{prop}\label{genus4notgor}
    The tautological ring $\RH^*(\M^{\ct}_{4,4})$ is not Gorenstein.
\end{prop}
 We follow the approach in Section \ref{3}, but we need a few new results about the cohomology of $\Mb_{4,4}$. Because all cohomology of $\Mb_{4,4}$ is tautological \cite[Theorem 1.4]{CL-CKgP}, we can write its Poincar\'e polynomial as
\begin{equation}\label{ppoly}
\sum_{i=0}^{26}(-1)^i \dim \H^i(\Mb_{4,4})t^i = \sum_{j=0}^{13}\dim \RH^{2j}(\Mb_{4,4})t^{2j}.
\end{equation}
By Poincar\'e duality, the polynomial \eqref{ppoly} is determined by $\dim \RH^{2j}(\Mb_{4,4})$ for $0\leq j\leq 6$. The Euler characteristic $\chi(\Mb_{4,4})$ determines one linear relation among the dimensions of the cohomology groups. It was calculated by Bini and Harer \cite[Table 2]{BiniHarer}.

\begin{lem}\label{Euler}
    We have $\chi(\Mb_{4,4})=327584$.
%\[
%\chi(\Mb_{4,4})%=43450s_4+56810s_{3,1}+26894s_{2^2}+19504s_{2,1^2}+1404s_{1^4}
%=327584\]
\end{lem}
To obtain two more linear relations, we use point counting results of Faber \cite{Faberemail}. Because the cohomology of $\Mb_{4,4}$ is pure Tate \cite[Theorem 1.4]{CL-CKgP}, its point count over $\mathbb{F}_q$ is given by substituting $t^2=q$ in the expression \eqref{ppoly}:
\[
\# \Mb_{4,4}(\mathbb{F}_q)  = \sum_{j=0}^{13}\dim \RH^{2j}(\Mb_{4,4})q^{j}.
\]
Faber has computed the point counts when $q=2,3$ \cite{Faberemail}. 
\begin{thm}[Faber]\label{pointcounts}
We have
\[
\#\Mb_{4,4}(\mathbb{F}_2)=48162399 \quad \text{ and } \quad \#\Mb_{4,4}(\mathbb{F}_3)=1392120592
.
\]
\end{thm}

Finally, using the pairing method in a computationally feasible range, we determine the dimensions of three of the cohomology groups.

\begin{lem}\label{lowdegreeterms}
    We have
    \[
    \dim \R\H^{2}(\Mb_{4,4})= 41, \quad \dim \R\H^{4}(\Mb_{4,4})= 589, \quad \dim \R\H^{6}(\Mb_{4,4})=4467,
    \]
    and the $3$-spin relations are complete in these cases.
\end{lem}
\begin{proof}
    Using \texttt{admcycles}, we compute 
    \[
    \dim \R_{\FZ}^1(\Mb_{4,4})=41, \quad \dim \R_{\FZ}^{2}(\Mb_{4,4})=589, \quad \dim \R_{\FZ}^{3}(\Mb_{4,4})=4467, 
    \]
    %and
    %\[
    %\dim (\S^{3}(\Mb_{4,4})/\mathsf{FZ}^{3}(\Mb_{4,4}))=4467,
    %\]
    giving upper bounds on $\dim \RH^{2i}(\Mb_{4,4})$ for $i=1,2,3$.
    To obtain lower bounds, we compute the rank of the pairings 
    \[
    \RH^{2i}(\Mb_{4,4})\times \RH^{26-2i}(\Mb_{4,4})\rightarrow \RH^{26}(\Mb_{4,4})\cong \qq
    \]
    as in Section \ref{yesgorensteinsection}. In each case $i=1,2,3$, the rank of the pairing agrees with the previously calculated upper bounds, so the $3$-spin relations are complete, and the dimensions of each group are given by the upper bounds.
\end{proof}

\begin{proof}[Proof of Proposition \ref{genus4notgor}]
The Poincar\'e polynomial of $\Mb_{4,4}$
\[
\sum_{j=0}^{13}\dim \RH^{2j}(\Mb_{4,4})t^{2j}.
\]
is determined by the coefficients of $t^{2j}$ for $0\leq j\leq 6$ and Poincar\'e duality. Lemma \ref{lowdegreeterms} gives the first four coefficients. Lemma \ref{Euler} and Theorem \ref{pointcounts} determine the values of the Poincar\'e polynomial at $t=1,2,3$, giving three linear relations among the coefficients. Solving the system, we see that $\dim \R\H^{10}(\Mb_{4,4})=52761$. Using \texttt{admcycles}, we have
\[
\dim \R_{\FZ}^5(\Mb_{4,4})\leq 52761,
\] which shows the $3$-spin relations are complete for $\R\H^{10}(\Mb_{4,4})$. By Lemma \ref{bartoct}, the $3$-spin relations are complete for $\RH^{10}(\M_{4,4}^{\ct})$. 

%Because the intersection pairing 
 %   \[
 %   \R\H^8(\M_{4,4}^{\ct})\times \R\H^{10}(\M_{4,4}^{\ct})\rightarrow \qq
 %   \]
  %  is $\ss_4$-equivariant, 
%It suffices to show that
%    \[
%    \dim \R\H^8(\M_{4,4}^{\ct})\neq \dim\R\H^{10}(\M_{4,4}^{\ct}).
%    \]
Using \texttt{admcycles}, we calculate%\footnote{The calculation for $\dim \R^4_{\FZ}(\M_{4,4}^{\ct})^{\ss_2\times\ss_2}$ is done over a finite field, which is the reason we obtain only an inequality.}
\[
    \dim \R_{\FZ}^4(\M_{4,4}^{\ct})\leq 6222 %2387 
    \quad \text{ and }\quad \dim \RH^{10}(\M_{4,4}^{\ct}) = \dim \R^5_{\FZ}(\M_{4,4}^{\ct}) = 6224.
    \]
Therefore, we have
 \begin{align*} \dim \R\H^8(\M_{4,4}^{\ct}) &\leq 
 \dim \R_{\FZ}^4(\M_{4,4}^{\ct}) \\
 &< \dim \R^5_{\FZ}(\M_{4,4}^{\ct}) \\
 &= \dim \R\H^{10}(\M_{4,4}^{\ct}),
 \end{align*}
so $\RH^*(\M^{\ct}_{4,4})$ is not Gorenstein.
%\[
%\dim \RH^8(\M_{4,4}^{\ct})\leq 6222
%\]
%\[
%\dim \RH^{10}(\M_{4,4}^{\ct})\leq 6224
%\]
%\samc{Calculation without taking invariants is still running. Would give us better control of size of Gorenstein kernel.}
%\samc{Calculated $\dim \RH^{10}(\M_{4,4}^{\ct})^{\ss_2}=3839$. Matrix was size $226669\times 114112$} \samc{Calculated $\dim \RH^{10}(\M_{4,4}^{\ct})=6224$}
\end{proof}

\subsection{Proof of Theorem \ref{Pixtontheorem}}
As in Section \ref{yesgorensteinsection}, we compute an upper bound for the dimensions of $\R^i(\M_{g,n}^{\ct})$ by calculating the rank of the matrix of $3$-spin relations modulo a prime. We calculate lower bounds by computing the rank of the intersection pairings
\[
\R^i(\M_{g,n}^{\ct})\times \R^{2g-3+n-i}(\M_{g,n}^{\ct})\rightarrow \qq.
\]
The upper and lower bounds agree and the pairing is perfect, except for when $i=\lfloor \frac{2g-3+n}{2}\rfloor$. Hence, the $3$-spin relations are complete when $i\neq \lceil \frac{2g-3+n}{2}\rceil$. For $(g,n)=(6,0)$ and $(5,2)$, we have also shown the $3$-spin relations are complete for $i=\lceil\frac{2g-3+n}{2}\rceil=5$ in the proof of Proposition \ref{5and6}. 

When $(g,n)=(7,0)$, we compute that the pairing
\[
\R^5(\M_7^{\ct})\times \R^6(\M_7^{\ct})\rightarrow \qq
\]
has rank $277$. This matches the upper bound for $\R^5(\M_7^{\ct})$ by calculated the matrix of $3$-spin relations. Therefore, $\dim \R^5(\M_7^{\ct})=277$. From calculating the rank of matrix of $3$-spin relations, we have $\dim \R^6(\M_7^{\ct})\leq 278$. By Lemma \ref{moregenus}, the kernel of the pairing is at least one dimensional, hence $\dim \R^6(\M_7^{\ct})=278$.  
\qed
\begin{rem}
    The only obstruction to extending Theorem \ref{Pixtontheorem} to a few more cases such as $(g,n)=(4,4)$ and $(g,n)=(3,6)$ is computational. We have not been able to carry out the necessary calculations in high codimension for these cases.
\end{rem}

\section{Computational aspects}\label{computationalaspects}
Many of the results above are based on extensive computer calculations using the software package \texttt{admcycles} \cite{admcycles}. The standard functions in the package (for computing ranks of tautological rings, intersection matrices, etc) are well-suited for computations with small $g,n$. However, some of the results above required working on spaces $\Mb_{g,n}$ and $\M_{g,n}^{\ct}$ outside this range, leading to a steep increase in combinatorial complexity. 

In the following sections we give a summary of some of these challenges and the mathematical and algorithmic adjustments that allowed us to finish the calculations. To avoid repetitions, fix some $g,n$ and denote by $\M$ either $\Mb_{g,n}$ or $\M^{\ct}_{g,n}$.

\subsection{Matrices of 3-spin relations}
A first step in many of the computations is calculating the $3$-spin relations $\FZ^r(\M) \subseteq \S^r(\M)$. For this, \texttt{admcycles} first enumerates all decorated strata $[\Gamma, \gamma]$ forming the basis of $\S^r(\M)$. The standard ordering of this basis lists the graphs $\Gamma$ in increasing number of edges, and each of them is then decorated in all possible ways with $\kappa$ and $\psi$-classes. For $\M=\M_{g,n}^{\ct}$ all decorations $\gamma = \prod_{v \in V(\Gamma)} \gamma_v$ with $\deg(\gamma_v)$ greater than the socle degree of $\R^*(\M_{g(v),n(v)}^{\ct})$ are omitted.

The relations themselves are stored as the rows of a matrix $M_\FZ$ whose columns correspond to the above basis of $\S^r(\M)$. The rows themselves are enumerated by tuples $T=(\Gamma_0, v_0, D, (\gamma_v)_{v \neq v_0})$ of
\begin{itemize}
    \item a stable graph $\Gamma_0$ of a stratum of $\M$ together with a choice of vertex $v_0 \in V(\Gamma_0)$ where a $3$-spin relation will be inserted,
    \item combinatorial data\footnote{We will not need the precise nature of $D$ in the following discussion, but see \cite{Pixton} for details.} $D$ determining which relation is glued into $v_0$,
    \item decorations $\gamma_v$ at all other vertices $v \in V(\Gamma) \setminus \{v_0\}$.
\end{itemize}
Again, the implementation in \texttt{admcycles} lists these tuples $T$ in increasing order of number of edges of $\Gamma_0$. Since the relation associated to $T$ is supported on decorated strata $[\Gamma, \gamma]$ such that $\Gamma_0$ is a contraction of $\Gamma$, the corresponding matrix $M_\FZ$ is in general a non-square matrix with roughly upper-triangular shape. For instance, relations associated to $T$ with $|E(\Gamma_0)| \geq e$ will not feature any decorated strata $[\Gamma, \gamma]$ with $E(\Gamma) < e$. As a result, the matrix $M_\FZ$ is in general quite sparse. We list a few examples of dimensions, ranks and densities of such matrices in Figure \ref{tab:FZ_matrix_examples}.

On a practical level, the matrix $M_\FZ$ is calculated entry by entry: there is a function \texttt{FZ\_coeff} which given the data of $T$ and a decorated stratum $[\Gamma, \gamma]$ computes the associated entry of $M_\FZ$ (in row $T$ and column $[\Gamma, \gamma]$).

\begin{figure}[htbp]
    \centering
    \begin{tabular}{|c|c|c|c|c|c|}
        \hline
        $\mathcal{M}$ & $r$ & $m_{\FZ}$ & $\dim \mathsf{S}^r(\mathcal{M})$  & $\rank M_{\FZ}$ & $\rho$ \\
        \hline
        % Add your data rows here
        % Example:
        $\mathcal{M}_{4,3}^{\ct}$ & 3 & 1052 & 1400 & 832 & 0.0368000 \\ % Matrix dimensions: 1052 x 1400 0.0367999728408474
        $\mathcal{M}_{4,3}^{\ct}$ & 5 & 67138 & 27359 & 26791 &  0.0048397 \\ % 0.00483968633801859
        $\mathcal{M}_{4,3}^{\ct}$ & 7 & 799508 & 154522 & 154502 & 0.0017108 \\ % 0.00171076878936427
        \hline
    \end{tabular}
    \caption{Some examples of moduli spaces $\M$ and Chow degrees $r$, for which we list the number $m_\FZ$ of rows and $\dim \mathsf{S}^r(\mathcal{M})$ of columns of the $3$-spin matrix $M_{\FZ}$, as well as its rank and density $\rho$ of non-zero entries}
    \label{tab:FZ_matrix_examples}
\end{figure}

Compared to the default implementation in \texttt{admcycles} we made the following optimizations for the project above:
\begin{itemize}
    \item The calculation of the rows of $M_\FZ$ was parallelized: there is one parent process enumerating the tuples $T$ indexing the rows of the matrix, which are distributed to a number of child processes which calculate the individual rows.
    \item The output of the function \texttt{FZ\_coeff} is calculated from different contributions (e.g. from various vertices) and the functions calculating these contributions will by default cache all of their previous results. While this speeds up the calculation, it also lead to significant memory blow-up. After analyzing this memory usage, we switched to a Least Recently Used (LRU) caching with a fixed number of cache entries. In practice LRU caching still provides moderate speedup while significantly reducing the memory profile of \texttt{FZ\_coeff}.
    \item Instead of storing the rows of the matrix in the working memory, they are saved to the disk storage using the \texttt{shelve} library in Python. This allows us to restart partial computations and to share the resulting $3$-spin matrices.
\end{itemize}

\subsection{Ranks and basis vectors}
After obtaining the relation matrix $M_\FZ$, we want to either
\begin{itemize}
    \item calculate its rank, to determine the conjectural dimension of the tautological ring as $\dim \S^r(\M) - \rank M_\FZ$,
    \item calculate a conjectural basis of $\R^r(\M)$ as a subset of the generators of $\S^r(\M)$, using the rows of $M_\FZ$ to eliminate such generators until we obtain a linearly independent set; in this case, we prefer to have a basis supported on graphs with few edges (since e.g. this makes it easier to calculate the intersection numbers in the next section).
\end{itemize}
In practice, both of these goals can be achieved by computing a row-reduced echelon form of $M_\FZ$, since the basis elements will correspond to columns of this echelon form \emph{not} given by pivots.

The default implementation of \texttt{admcycles} performs this echelonization over the rational numbers (since the entries of $M_\FZ$ are by default elements of $\mathbb{Q}$). This has the advantage that the resulting echelon form allows us to express arbitrary generators of $\S^*(\M)$ in terms of the conjectural basis of $\R^*(\M)$, and this is used e.g. when comparing tautological classes. The two draw-backs of working over $\mathbb{Q}$ are that
\begin{itemize}
    \item divisions during the elimination process generally lead to a blow-up of denominators. While \texttt{SageMath} can calculate with rational numbers having numerator and denominator of arbitrary size, this can require significant memory usage.
    \item many specialized libraries for (sparse) linear algebra are geared towards calculations over finite fields.
\end{itemize}
In practice, choosing a random mid-sized prime (like $p=4001$), we can convert the matrix $M_\FZ$ to a matrix $M_\FZ^p$ over $\mathbb{F}_p$. This conversion only uses that none of the denominators of any entry are divisible by $p$. The expectation is that both the rank and the pivot columns of the echelon form of the matrices $M_\FZ$ and $M_\FZ^p$ coincide. 

More formally, it is the case that reducing mod $p$ can only \emph{lower} the rank of the matrix, which would make us \emph{miss} a tautological relation. When calculating the pairing matrices in the next step, this would produce an unexpected element in their kernel (representing the missed relation). Thus, if the intersection matrix does have full rank, we can \emph{a posteriori} conclude that the reduction modulo $p$ did not change the rank of $M_\FZ$.

Most calculations in the paper above were performed using the sparse echelonization algorithms of the \texttt{LinBox} library \cite{dumas2002linbox}. 

Here there is a subtle phenomenon: the standard row-reduced echelon form chooses the lexicographically smallest set of pivot columns in the row reductions. 
When applying it to the matrix $M_\FZ^p$ this benefits from the roughly right-upper-triangular nature of the matrix, since the elimination process does not fill too many of the zero entries of the matrix during intermediate steps. 
However, we do encounter a problem when trying to compute a basis of $\R^r(\M)$ consisting of decorated graphs with few edges. For this we need to find the lexicographically \emph{largest} set of pivot columns, essentially running the echelonization on a vertical flip of the original matrix $M_\FZ^p$. Now the matrix is left-upper triangular, which leads to significantly more fill-in in the intermediate stages of the algorithm. For this reason, we sometimes  manage to calculate the conjectural rank of $\R^r(\M)$ but fail to obtain a candidate basis.

\subsection{Intersection pairings}
Assume that in the previous step we managed to calculate that $\S^r(\M)/\FZ^r(\M)$ has dimension $d_r$. Then one way to show that the set $\FZ^r(\M)$ of $3$-spin relations is complete is to show that the intersection pairing
\[
\S^r(\M) \times \S^{r_c}(\M) \to \mathbb{Q}
\]
has rank $d_r$, where $r_c = \mathrm{socdeg}(\M)-r$ is the complementary degree to $r$. In principle this is a finite calculation, but again there are several possible speedups:
\begin{itemize}
    \item If via the previous step we were able to obtain conjectural bases for $\R^r(\M)$ and $\R^{r_c}(\M)$ then we can just calculate the pairing matrix (which is expected to be densely filled with entries in $\mathbb{Q}$) and compute its rank. This calculation can also be done modulo $p$, since a lower bound on the rank is enough.
    \item If we have a conjectural basis $\mathcal{B}$ for $\R^r(\M)$ but not for $\R^{r_c}(\M)$, we can instead follow a heuristic algorithm: we iteratively choose generators of $
    \S^{r_c}(\M)$, compute the vectors of intersection numbers with elements of $\mathcal{B}$ and add them as rows to a matrix $I$. By performing row-reductions on $I$ from time to time, we can monitor its rank, and the calculation finishes when this rank becomes maximal (equal to the cardinality $d_r$ of $\mathcal{B}$). In practice, one can start with picking random generators of $\S^{r_c}(\M)$ with at most $e_0 \geq 0$ edges, and increase the bound $e_0$ if the rank of $I$ starts stabilizing.
    \item If only the rank of $\R^r(\M)$ is known, one can apply the above heuristic by similarly choosing increasing sets of generators of $\S^r(\M)$ and $\S^{r_c}(\M)$ and tracking the rank of the resulting square matrix. The disadvantage is that previous echelonizations of intersection matrices $I$ cannot obviously be used in the next step since both rows and columns are added to $I$. 
\end{itemize}
Again, since calculations of intersection numbers are logically independent, some speed-up via parallelization is possible when calculating the entries of the matrix $I$.

\bibliographystyle{amsplain}
\bibliography{refs}
\end{document}